\documentclass{amsart}
\usepackage{graphicx,amsmath,amssymb,curves,color}
\usepackage[pdftex]{hyperref}
\begin{document}

\newtheorem{thm}{Theorem}[section]
\newtheorem{cor}[thm]{Corollary}
\newtheorem{lem}[thm]{Lemma}
\newtheorem{prop}[thm]{Proposition}
\renewcommand\[{\begin{equation}}
\renewcommand\]{\end{equation}}

\theoremstyle{remark}
\newtheorem{notations}[thm]{Notations}
\newtheorem{notation}[thm]{Notation}

\theoremstyle{definition}
\newtheorem{defn}[thm]{Definition}
\newtheorem{defns}[thm]{Definitions}

\theoremstyle{remark}
\newtheorem{rem}[thm]{Remark}
\newtheorem{eg}[thm]{Example}
\newtheorem{ob}[thm]{Observation}
\numberwithin{equation}{section}
\newcommand{\norm}[1]{\left\Vert#1\right\Vert}
\newcommand{\abs}[1]{\left\vert#1\right\vert}
\newcommand{\set}[1]{\left\{#1\right\}}
\newcommand{\Real}{\mathbb R}
\newcommand{\C}{\mathbb{C}}
\newcommand{\Rational}{\mathbb{Q}}
\newcommand{\Pro}{\mathbb{P}}
\newcommand{\eps}{\varepsilon}
\newcommand{\To}{\longrightarrow}
\newcommand{\TTo}{\dashrightarrow}
\newcommand{\BX}{\mathbf{B}(X)}
\newcommand{\Z}{\mathbb{Z}}
\newcommand{\K}{\mathbb{K}}
\newcommand{\tb}{\textbf}
\newcommand{\mc}{\mathcal}
\newcommand{\mr}{\mathrm}
\newcommand{\mf}{\mathbf}
\newcommand{\mb}{\mathbb}
\newcommand{\ta}{\theta}
\newcommand{\al}{\alpha}
\newcommand{\be}{\beta}
\newcommand{\TC}{(\mathbb{C}^*)}
\newcommand{\TK}{(\mathbb{K}^*)}
\newcommand{\Sev}{\mathrm{Sev}(\Delta,\delta )}
\newcommand{\TS}{\mathrm{Trop}(\mathrm{Sev})}
\newcommand{\p}{\phi}
\newcommand{\w}{\omega}
\newcommand{\n}{\mathrm{in}_}
\newcommand{\D}{\Delta}
\newcommand{\Tr}{\mathrm{Trop}}
\newcommand{\bs}{\boldsymbol}
\newcommand{\bsm}{\boldsymbol{m}}
\newcommand{\SD}{\mathcal{S}(\Delta)}
\newcommand{\TSD}{\mathcal{T}\mathcal{S}(\Delta)}
\newcommand{\T}{\mathbb{T}}

\title[]{Tropical Severi Varieties}%

\author{Jihyeon Jessie Yang}
\address{Department of Mathematics and
Statistics\\ McMaster University\\ 1280 Main
Street West\\ Hamilton, Ontario L8S4K1\\ Canada}
\email{jyang@math.mcmaster.ca}

\thanks{We would like to thank the referees for careful reading of the manuscript
and giving numerous helpful suggestions.
}

\begin{abstract}
We study the tropicalizations of Severi
varieties, which we call {\em tropical Severi
varieties}.  In this paper, we give a partial
answer to the following question,  ``describe the
tropical Severi varieties explicitly.'' We obtain
a description of tropical Severi varieties in
terms of regular subdivisions of polygons. As an
intermediate step,  we construct  explicit
parameter spaces of  curves. These parameter
spaces are much simpler objects than the
corresponding Severi variety and they are closely
related to flat degenerations of the Severi
variety, which in turn describes the tropical
Severi variety. As an application, we understand
G.Mikhalkin's correspondence theorem for the
degrees of Severi varieties in terms of  tropical
intersection theory. In particular, this provides
a proof of the independence of
point-configurations in the enumeration of
tropical nodal curves.
\end{abstract}

\maketitle

\section{Introduction}

The advent of tropical geometry and its fast
development suggest to look at the classical
algebraic geometry in a different perspective.
Tropicalization is an operation that turns
subvarieties of an algebraic torus  into
polyhedral objects in  a real vector space  along
with a locally-constant integral-valued function
on it.  This procedure enables us to build an
intersection theory on the algebraic torus called
\emph{tropical intersection theory} which can be
used to solve classical enumerative questions.
There are many attractive properties of this
intersection theory. First of all, we work with
polyhedral objects instead of algebraic
varieties. Also we sometimes do not need to
consider compactifications.
The case of hypersurfaces is  closely related to
Newton polytope theory: the tropicalization of a
hypersurface defined by a Laurent polynomial $f$
is the codimension $1$ skeleton of the
outer-normal fan of
 the Newton polytope of $f$. Moreover, the locally-constant
integral-valued   function on it  is determined
by the lattice lengths of the edges of the Newton
polytope of $f$.
Using tropical geometry, G. Mikhalkin\cite{Mikhalkin}
found a purely combinatorial method to compute
the degree of a Severi variety (or Gromov-Witten
invariants of plane $\Pro^2$). This celebrated
work, \emph{correspondence theorem}
(Theorem \ref{thm:correspondence}),  has brought tropical
methods to the attention of geometers and
motivated active systematic developments of
tropical geometry.
A Severi variety $\Sev$ (Definition \ref{defn:Severi}) is a complex projective
variety which parameterizes  curves on the toric
surface $X_{\D}$ of an integral polygon $\D$
with a given number $\delta$ of nodal singularities. It is
known that the degree of a Severi variety is
equal to the number of such nodal curves passing
through a certain number of generic points in
$\C^2$. Mikhalkin's correspondence theorem,
simply speaking, says that this enumerative
number is equal to the number of \emph{tropical}
plane curves passing through the same number of
generic points in $\Real^2$  counted with certain
multiplicities, called Mikhalkin's multiplicity
(\S\ref{sec:Severi degree}).  These tropical plane
curves are tropicalizations of classical
algebraic curves and they are purely
combinatorial objects. Therefore, the enumerative
problem becomes a purely combinatorial one.
 This approach  brings up some questions such as, ``why does
this count give the solution for the original
classical problem?'' and ``what is the meaning of
such multiplicities assigned to tropical
curves?''. The answers for these questions are
given in the proof of Mikhalkin's correspondence
theorem \cite{Mikhalkin} or Shustin's proof based
on algebraic geometry \cite{Shustin2}. However,
their proofs on the independence of
point-configuration relies on the well-known fact
 that the number of  {\em classical} nodal curves
 does not depend on the position of points, which is
 a
fact from classical intersection theory. Another
proof of the independence of point-configuration
 is
proved in \cite{GM} by considering  moduli spaces
of tropical curves.

In this paper,  we take another approach: we
study the tropicalization of a
 Severi variety $\Tr(\Sev)$ and understand the enumeration
of nodal curves in terms of tropical intersection
theory. The subset $\Tr(\Sev)$ of a  real vector
space has a   natural weighted  fan  structure
induced from the Gr\"{o}bner fan  of $\Sev$. The
main ingredients of the definition of $\Tr(\Sev)$
are the initial schemes $\n\w\Sev$ of the very
affine Severi variety, $\Sev^{\circ}$, which is
the intersection of the Severi variety $\Sev$ and
the big open torus of the ambient projective
space of $\Sev$. When $\w$ is a {\em regular}
point of $\Tr(\Sev)$, that is, it is in the
relative interior of a maximal cone of
$\Tr(\Sev)$, we know that the initial scheme
$\n\w\Sev$ is supported on the union of finitely
many translates of a subtorus of the big torus
(for example, see \cite[\S6.]{Katz2}). The number
of such translations of a subtorus (counted with
multiplicities) is called the {\em weight} of
$\w$.
 The first simple description of (the support of) $\Tr(\Sev)$ is
given by  a positive integer $\mr{rank}(\w)$
(Definition \ref{defn:rank}) assigned to every
point in the ambient real vector space of
$\Tr(\Sev)$:

\begin{thm}[Theorem \ref{thm:support}]
If an integral point $\w$ has $\mr{rank}(\w) >
\mr{dim}(\Sev)$, then $\w$ is not in $\Tr(\Sev)$.
\end{thm}

The following two descriptions of $\Tr(\Sev)$
uses the regular subdivisions of polygons (\S
\ref{sec:subdivision}).

\begin{thm}[Theorem \ref{thm:No nonprimitive}]
Let $\w$ be an integral point in  $\Tr(\Sev)$
with  the following conditions:
\begin{itemize}
\item $\mr{rank}(\w)=\mr{dim}(\Sev)$;
\item The regular subdivision $\D_{\w}$ has no non-primitive parallelogram.
\end{itemize}
Then $\w$ is a regular point of $\Tr(\Sev)$, that
is, $\w$ is in a maximal cone of $\Tr(\Sev)$.
Furthermore, the weight of $\w$ on $\Tr(\Sev)$ is
equal to

\[\bsm_{\Sev}(\w)=l(\mb{V})\cdot
\widetilde{\prod}\mr{length(\mr{Edges}(\D_{\w}))},\]
where
\begin{enumerate}
\item $l(\mb{V})$ is the number of connected components of $\mb{V}$;
\item $\widetilde{\prod}\mr{length(\mr{Edges}(\D_{\w}))}$ is the product of the
lattice lengths of the  edges  which are
representatives of each equivalence class in
$\mr{Edges}(\D_{\w})$, where we define an equivalence
relation as follows: let $e\sim e'$ if $e$ and
$e'$ are the parallel edges of a parallelogram in
$\D_{\w}$ and extend it by transitivity.
\end{enumerate}

\end{thm}
\begin{thm}[Theorem \ref{thm:Existence of
nonprimitive}]
Let $\w$ be an integral point in  $\Tr(\Sev)$with
the following conditions:
\begin{itemize}
\item $\mr{rank}(\w)=\mr{dim}(\Sev)$ ;
\item $\w$ is a regular point in $\Tr(\Sev)$.
\end{itemize}
Then the weight of $\w$ on $\Tr(\Sev)$ is equal
to
\[\bsm_{\Sev}(\w)=l(\mb{V})\cdot
\widetilde{\prod}\mr{length(\mr{Edges}(\D_{\w}))},\]
as defined in the previous theorem.
\end{thm}
Using these results on tropical Severi varieties,
we provide another proof of the independence of
the configuration of points in the enumerations
of tropical plane curves (\S\ref{sec:Severi
degree}).

\section{Preliminaries}  In tropical geometry, we study {\em tropical} varieties, which are
 polyhedral objects with certain properties. The precise definition is presented in
 \S \ref{sec:Tropical varieties}.  Classical algebraic
 varieties are connected with tropical varieties by an operation
 called \emph{tropicalization}. In this section, we review only what we need to
study
  parameter
 spaces of algebraic curves on toric surfaces.

 \subsection{Tropical Varieties}
 \label{sec:Tropical varieties}\cite{MS}
Let $\Sigma\subset\Real^n$ be a one-dimensional
fan with $r$ rays. Let $\bs{u}_i$ be the first
lattice point on the $i$th ray of $\Sigma$. We
give $\Sigma$ the structure of a \emph{weighted
fan} by assigning a weight $m_i\in\mb{N}$ to the
$i$th ray of $\Sigma$.
\setlength{\unitlength}{0.2mm}
\begin{picture}(0,0)(-80,40)
\put
(0,0){\vector(1,0){30}}\put(33,-3){\tiny{$2$}}\put(10,-4)
{\tiny{$\bullet$}}\put(10,
-10){\tiny{$\bs{u}_1$}}
\put(0,0){\vector(-1,1){25}}\put(-36,25){\tiny{$1$}}
\put(-15,8){\tiny{$\bullet$}} \put(-6,
9){\tiny{$\bs{u}_2$}}
\put(0,0){\vector(-1,-2){24}}\put(-30,
-60){\tiny{$1$}}\put(-11,-20){\tiny{$\bullet$}}\put(-27,
-13){\tiny{$\bs{u}_3$}}

\end{picture}
We say that $\Sigma$ is \emph{balanced} if
\[\sum_{i=1}^r m_i\bs{u}_i=0.\]
We now extend this concept to arbitrary
polyhedral complexes.

\begin{defn}
 Let $\Sigma\subset\Real^n$ be a rational
 polyhedral fan, of pure dimension $d$, and
 fix $m(\sigma)\in\mb{N}$ for all maximal cones
 $\sigma$.
 Such $\Sigma$ is called a  \emph{weighted} fan.
We denote the set of all cones of dimension $k$
in $\Sigma$ by $\Sigma^{(d-k)}$. Let
$\tau\in\Sigma^{(1)}$ and let $L_{\tau}$ be the
affine span of $\tau$. Note that since $\tau$ is
a {\em rational} cone, $L_{\Z}=L\cap\Z^n$ is a
free abelian group of rank $d-1$ with
$\Z^n/L_{\Z}\cong\Z^{n-d+1}$. For each maximal
cone $\sigma\in\Sigma$ with $\tau\subset\sigma$
the cone $(\sigma+L_{\tau})/L_{\tau}$ is a
one-dimensional cone (ray) in $\Real^n/L_{\tau}$.
Let $\bs{u}_{\sigma/\tau}$ be the first lattice
point on this ray.
The weighted fan $\Sigma$ is \emph{balanced at
$\tau$} if
\[\sum m(\sigma)\bs{u}_{\sigma/\tau}=0 \in
\Real^n/L_{\tau},\] as the sum varies over all
maximal cones containing $\tau$.
\setlength{\unitlength}{0.15mm}
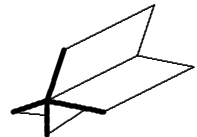
\begin{picture}(0,0)(-140,0)
\curve(15,45,95,85)\curve(0,0,80,40)\curve(80,40,95,85)
\curve(0,0,50,-10)\curve(50,-10,130,30)\curve(80,40,130,30)
\curve(0,-13,-30,-10)\curve(0,0,-30,-10)\curve(0,-30,35,-7)
\linethickness{0.8mm}
\curve(0,0,15,45)\curve(0,0,0,-30)
\curve(0,0,-30,-10)\curve(0,0,50,-10)
\end{picture}\\\\
The weighted fan $\Sigma$ is \emph{balanced} if
it is balanced at all $\tau\in\Sigma^{(1)}$.
Now let $\Sigma$ be a  rational polyhedral
complex of pure dimension $d$ with weight
$m(\sigma)\in\mb{N}$ on each
 maximal polyhedron $\sigma$ in $\Sigma$. Then for each $\tau\in\Sigma$ the fan $\mr{star}_{\Sigma}(\tau)$ inherits a
 weighting function $m$, where $\mr{star}_{\Sigma}(\tau)$ is the \emph{star} of $\tau\in\Sigma$ whose cones are indexed by those
 $\sigma\in\Sigma$ for which $\tau$ is a face of $\sigma$: Fix $\w\in\tau$.
 Then the cone of $\mr{star}_{\Sigma}(\tau)$ indexed by $\sigma$ is the Minkowski sum
 \[\bar{\sigma}:=\{v\in\Real^n: \exists \epsilon >0 \text{ with } \w+\epsilon v\in \sigma\} + \text{aff}(\tau)-\w,\] where $\text{aff}(\tau)$ is the affine span of $\tau$. This is independent of the choice of $\w$.
The weighted polyhedral complex $\Sigma$ is \emph{balanced} if the fan
 $\mr{star}_{\Sigma}(\tau)$
 is balanced for all $\tau\in\Sigma^{(1)}$.
\end{defn}

\begin{defn}\cite{AR},\cite{Kazarnovskii2}, \cite[\S 1]{ST}\begin{enumerate}
\item A  {\em homogeneous tropical variety} of {\em degree} $n-d$
 is a pair $(\mc{T},\bsm)$,
where $\mc{T}$ is a subset of $\Real^n$ and
$\bsm:\mc{T}^{\circ}\rightarrow \Z_{>0}$ is a
locally constant function, called \emph{weighting
function} which satisfies:
\begin{itemize}\item There exists a pure $d$-dimensional rational polyhedral complex supported on
$\mc{T}$;
\item $\mc{T}^{\circ}\subset\mc{T}$ is the open subset of regular points, where
$\w\in\mc{T}$ is called \emph{regular} if there
exists a vector subspace $L_{\w}\subset\Real^n$
such that  locally near $\w$, $\mc{T}$ is equal
to a translation $L_{\w}+v$ of $L_{\w}$ for some
$v\in\Real^n$.
\item The function $\bsm$ satisfies the  balancing condition  for
one (and hence for any) polyhedral complex
supported on the set $\mc{T}$.
\end{itemize}
\item A {\em tropical cycle} is a formal sum of homogeneous
tropical varieties of different degrees.

\end{enumerate}
\end{defn}

\subsection{Tropicalization}\cite{Katz2}, \cite{
Kazarnovskii2},
\cite{MS}\label{sec:Tropicalization}

 We  consider an extension of the complex field $\C$: Let
$\K$ denote  the field of locally convergent
Puiseux series over $\C$, that is, the elements
of $\K$ are power series of the form
\[b(t)=\sum_{\tau\in R}c_{\tau}t^{\tau},\] where $R\subset\Rational$
is contained in an arithmetic progression bounded
{\em from above}, $c_{\tau}\in\C$,  and
$\sum_{\tau\in R}|c_{\tau}|t^{\tau}<\infty$ for
sufficiently large positive $t$. This is an
algebraically closed field of characteristic zero
with a non-Archimedean valuation
\[Val(b)=\mr{max}\{\tau\in R:c_{\tau}\ne 0\}.\] Without loss of generality we may suppose that
$Val(b)$ is an integer by changing the parameter
$t\mapsto t^l$ for some $l$. We always assume
this unless mentioned otherwise.
Note that this definition of $\K$ is slightly
different from the one in \cite{Shustin2}.
However, the one in \cite{Shustin2} can be
obtained by the substitution, $t\mapsto t^{-1}$.
Also the tropicalization can be defined for a
general field with a non-Archimedean valuation.
However, it is enough to consider only $\K$ for
our purpose of the study of parameter spaces of
curves on toric surfaces.

 By a \emph{scheme}
we shall mean an algebraic scheme over the  field
$\K$, that is, a scheme $X$ together with a
morphism of finite type from $X$ to $\mr{Spec}(\K)$. A
{\em variety} will be a reduced scheme, and a
{\em subvariety} of a scheme will be a closed
reduced subscheme. A {\em point} on a scheme will
always be a closed point. A {\em curve} is a
1-dimensional scheme. In fact, the
tropicalization is an operation which is defined
only on subvarieties of an algebraic torus
$\TK^n=\mr{Spec}(\K[\Z^n])$. The subvarieties of
$\TK^n$ are called {\em very affine varieties}.
We  often use  the vector-notation, for example,
$c^a$ is $c_1^{i}\cdot c_2^{j}$ and $c_a$ is
$c_{(i,j)}$ where $a=(i,j)$. Also we  often
identify a Laurent polynomial with a function on
a finite set as follows: let
$f=\sum_{a\in\mc{A}}c_ax^a$ be a Laurent
polynomial with variables $x$ and coefficients
$c_a$ in a ring $\mc{R}$, where $\mc{A}$ is a
finite subset of $\Z^n$. Then $f$  is identified
with the function \[f:\mc{A}\rightarrow
\mc{R},\quad a\mapsto c_a.\]


\subsubsection{Tropicalization: Varieties over $\C$.}\label{sec:constant}
In this section we present a precise definition
of the tropicalization of $X$, where $X$ is a
subvariety of an algebraic torus over $\C$,
$\T=\mr{Spec}(\C[\Z^n])=\mr{Spec}(\C[x_1^{\pm 1},\dots,
x_n^{\pm 1}])$.
\begin{defns}$ $\label{defns:InitialVersion1}
\begin{itemize}
\item (initial sets) Given a finite subset $\mc{A}$ of $\Z^n$ and a vector $\w\in\Z^n$, let $\n\w
\mc{A}$ denote the set of all points $a\in
\mc{A}$ such that the inner product $a\cdot \w$
is maximal.
\item (initial polynomials) Given a Laurent polynomial $f\in \TC^{\mc{A}}$ and
a vector $\w\in\Z^n$,  the {\em $\w$-degree} of
$f$ is the maximum of $a\cdot  \w$ for all $a\in
\mc{A}$.   Let $\n\w f$ denote the restriction of
$f$ to $\n\w \mc{A}$ so that  $\n\w f$ is
homogeneous with respect to $\w$-degree.  We can
also see that
\[\n\w f (x)= \mr{lim}_{t\rightarrow \infty} t^{-\gamma}f(t^{\w}\cdot
x),\] where $\gamma$ is the $\w$-degree of $f$.
\item (initial ideals) Given an ideal $I$ in the ring of Laurent polynomials $\C[\Z^n]$ and a
 vector
$\w\in \Z^n$, let $\n\w I$ denote the ideal
generated by all initial polynomials $\n\w f$ for
$f\in I$.
\item (initial schemes) Given an affine scheme $V(I):=\mr{Spec}(\C[\Z^n]/I)$ and a
vector $\w\in\Z^n$, the scheme $V(\n\w
I):=\mr{Spec}(\C[\Z^n]/\n\w I)$ defined by $\n\w I$ is
called {\em the initial scheme} of
 $V(I)$ with respect to $\w$.
\end{itemize}
\end{defns}

\begin{rem}\label{rem:Flat}
We may consider the initial scheme $V(\n\w I)$ as
a flat degeneration of $V(I)$, that is, there
exist a one-parameter flat family
 $V(I_t)$ such that $V(I)=V(I_1) $ and $V(\n\w I)=V(I_0)$. The details are given in
 \cite[\S 15.8]{Eisenbud}.
\end{rem}

\begin{defns}\label{defn:cvectorVersion1} Let $X$ be an irreducible  subvariety
of the algebraic torus $\T=\T_{\C}$.

\begin{itemize}
\item A vector $\w\in\Z^n$ is called a
\emph{c-vector} of $X$ if the initial ideal
  $\n\w I_X$ of the defining ideal $I_X$ of $X$ contains no monomial, equivalently, the
initial scheme $\n\w X:= V(\n\w I_X)$ is not
empty. (``c'' is the first letter of ``current''
as introduced in \cite{Kazarnovskii2})
\item The (support of the) tropicalization of $X$  is the closure of the union of positive rays
$\Real_{\ge 0}\cdot\w$  generated by all
c-vectors $\w$ of $X$ and it is denoted by
$\mathrm{Trop}(X)$.
\item We define a weighting function $\bsm_X$ on $\Tr(X)$ as follows:
 For a point   $\w$ in $\mr{Trop}(X)$ let $\bs{m}(\w)$  to be the sum of the
 multiplicities of all minimal associated prime ideals of the initial ideal
 $\n\w I_X$.

 \item The set $\Tr(X)$ together with the weighting function
 $\bsm_X$ is called the {\em tropicalization } of $X$ and it is
 denoted again by $\Tr(X)$.
\end{itemize}
\end{defns}


\begin{eg}\label{eg:TropicalConic}Let
$X=V((1+x+y)^2)$ be the plane curve. The picture
below is the support of $\mr{Trop}(X)$. The
integers near the three rays are the
corresponding weights. For example, let
$\w=(0,-1)$. The corresponding initial
polynomial, $\n{(0,-1)}f$, is equal to $(1+x)^2$
and so $\bsm((0,-1))=2$.\end{eg}

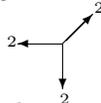
\begin{picture}(0,0) (-350, -10)
\put(50,10){\vector(-1,0){30}}
\put(50,10){\vector(0,-1){30}}\put(50,10){\vector(1,1){20}}
\put(13,8){\tiny{$2$}}
\put(48,-30){\tiny{$2$}}\put(71,31){\tiny{$2$}}
\end{picture}
%
%
\subsubsection{Tropicalization}\label{subsec:Version2}
Now we consider the general case: varieties over
$\K\supset\C$. As noted in Remark \ref{rem:Flat},
for varieties defined over $\C$, tropicalization
is a way to record certain flat limits.
Geometrically, the extension of field from $\C$
to $\K$ can be seen as adding one parameter $t$
and considering  flat families of varieties. This
idea is well described in Proposition
\ref{prop:Equisingular} in the case of curves on
toric surfaces. By generalizing  the definition
of the initial ideal to consider the role of the
new parameter $t$, we can define  the
tropicalization of a variety over $\K$ in a
straightforward way.

\begin{defns}$ $\label{defns:initial}
\begin{itemize}
\item (initial sets) Let $(\mathcal{A},\nu)$ be a pair of  a finite subset $\mathcal{A}$ of $\Z^n$ and a
       real-valued function $\nu$ on it. Given a $(\mathcal{A},\nu)$
      and a vector $\w$ in $\Z^n$, let $\n\w (\mathcal{A},\nu)$ denote the set of all points $a\in
    \mathcal{A}$ such that  $a\cdot \w+\nu(a)$ is maximal.
\item (initial polynomials) Given a Laurent polynomial
$f=\sum_{a\in\mc{A}}c_a(t)x^a$ over $\K$  and a
vector $\w\in\Z^n$, the  {\em $t-\w$-degree } of
$f$ is the maximum of $a\cdot\w+Val(c_a)$ for all
$a\in \mc{A}$. We define $\n\w f$ which is
defined over $\C$ as follows: \[\n\w f(x)=
\mr{lim}_{t\rightarrow\infty}t^{-\gamma}f(t^{\w}\cdot
x),\] where $\gamma$ is the $t-\w$-degree of $f$.
The following shows how to obtain $\n\w f$:
Let $\Delta=\mr{Newton}(f)$ be the Newton polytope of
$f$. We take the convex hull $\widetilde{\Delta}$
of the set $\{(a,Val(c_a))\in \Z^{n+1}: a\in
\mathcal{A}\}$ and introduce the function
\begin{equation}\nu_f:\Delta\rightarrow \Real,\quad
\nu_f(\alpha)=\mr{max}\{\beta:(\alpha,\beta)\in\widetilde{\Delta}\}.\end{equation}\label{eq:50}
This is a concave piecewise-linear function.
Notice that we can write
\begin{equation}c_a(t)=c_a^{\circ}\text{} t^{\nu_f(a)}+l.o.t.,\end{equation}\label{eq:11} where  $c_a^{\circ}$ is some complex
number which is zero if $\nu_f(a)>\mr{Val}(c_a)$.
 Then
\[\n\w f=\sum_{a\in \n\w(\mathcal{A},\nu_f)}c_a^{\circ}\text{} x^a\]

\item (initial ideals) For an ideal $I$ in the ring of Laurent
polynomials $\K[\Z^n]$ and a vector $\w\in \Z^n$,
let $\n\w I\subset\C[\Z^n]$
 denote
 the ideal
generated by all initial polynomials  $\n\w f$
for $f\in I$.
\end{itemize}
\end{defns}
The definition of the tropicalization given in
the previous subsection \S \ref{sec:constant}
 generalizes in a straightforward way for an irreducible subvariety $X$
 of the algebraic torus $\T_{\K}$.\\
 Now the proofs of the following facts can be found in many references
on tropical geometry, for example, \cite{BG},
\cite{EKL}, \cite{MS}:

\begin{itemize}
\item The tropicalization $\Tr(X)$ has the structure of a homogeneous
tropical variety with dimension  $\mr{dim}(X)$
\item Suppose  $X$ is defined over $\C$, i,e., the ideal
$I_X\subset\K[\Z^n]$ of $X$ can be generated by
Laurent polynomials over $\C$ and thus $X$ is
independent of the parameter $t$. (Sometimes,
this case is called {\em constant coefficient
case}.) Then $\mr{Trop}(X)$ coincides with
$\mr{Trop}(X(\C))$, that is, we can treat it as
in the previous section. Furthermore, in this
case $(\Tr(X), \bsm_X)$ has a balanced polyhedral
{\em fan} structure.
\item For any regular point $\w\in\Tr(X)$, the initial scheme $\n\w X$ is a
union of finitely many translations of a subtorus
$\mb{G}^e$ of $\T=\T_{\C}$.
\end{itemize}

\subsection{Tropical plane curves}
We want to study   complex algebraic curves on
toric surfaces by studying their
tropicalizations. Recall that tropicalization is
defined for subvarieties of algebraic tori.
Thus we only consider the intersections of the
curves with the big open torus $\TC^2$ of the
toric surfaces. It is very easy to understand the
tropicalizations of the curves: any such curve
 $X$ in $\TC^2$ is defined by a Laurent polynomial $f(x,y)$
in two variables over $\C$. The tropicalization
$\mathrm{Trop}(X)$ is the 1-dimensional skeleton
of the (outer) normal fan of the Newton polygon
of $f$ and the multiplicity of each ray is the
lattice length of the corresponding edge of the
Newton polygon of $f$ (see Example
\ref{eg:TropicalConic}). The problem is that
$\mathrm{Trop}(X)$ does not distinguish much
about the  curves because curves with the same
Newton polygon have the same tropicalization.
This lack of information is overcome by
considering a more generalized version of
$\mathrm{Trop}(X)$. Namely, we extend the field
of coefficients from $\C$ to $\K$, the field of
locally convergent Puiseux series over $\C$.
The relationship between algebraic  curves over
$\C$ and $\K$ is formulated in the following
statement.
\begin{prop}\cite[\S 2.3.]{Shustin2} \label{prop:Equisingular} Let  a Laurent polynomial
\[f(x,y)=\sum_{(i,j)\in\Delta} c_{ij}(t)x^iy^j\in\K[x,y]\] define a curve $C_{\K}\subset
\TK^2$ with only isolated singularities and
Newton polygon $\Delta$. We obtain a
one-parameter  family $C^{(t)}$ of curves on
$\TC^2$. For $t$ with $|t|>> 0$, the family
$C^{(t)}$ is equisingular and the topological
types of singularities of $C^{(t)}$ are in
$1$-to-$1$ correspondence with topological types
of singularities of $C_{\K}$.
\end{prop}

Shustin \cite{Shustin2} found  nice
characterizations of the tropicalizations of
nodal curves on toric surfaces which are crucial
to the study of the initial schemes of Severi
varieties. However, these characterizations are
valid only for nodal curves satisfying certain
base-point conditions. In the following
subsections we show that how we can replace the
base-point conditions by the rank-condition
(Definition \ref{defn:rank}).

\subsubsection{Subdivisions of Polygons and
Adjacency graphs.}\label{sec:subdivision}

\begin{defns}Let $\D$ be a convex lattice polygon. \begin{enumerate}
\item A {\em subdivision} of  $\D$, denoted by $\SD$, is a decomposition
of $\D$ into a finite number of non-degenerate
convex lattice sub-polygons   such that the
intersection of any two of these sub-polygons is
a common face of  both of them (maybe empty). (We
consider $\D$ itself as a subdivision of $\D$
with one two-dimensional face.)
\item Given a subdivision $\SD$, let $\mr{Vertices}(\SD), \mr{Edges}(\SD),
\mr{Faces}(\SD)$,\\ $\mr{\mr{Triangles}}(\SD),
\mr{Parallelograms}(\SD), \mr{Int}(\SD)\cap\Z^2$ be the set
of vertices, edges, (2-dimensional) faces,
triangles, parallelograms , interior lattice
points of $\SD$, respectively.
\item A subdivision $\SD$ is called
\begin{itemize}\item {\em triangular} if every 2-dimensional face  is a
triangle;
\item {\em nodal} if every 2-dimensional face is either a triangle or a
parallelogram; \item {\em simple} if every
lattice point on the boundary of $\D$ is a vertex
of $\SD$.
\end{itemize}
\item A subdivision $\SD$  is called \emph{regular} if there exists a continuous
concave piecewise-linear  function on $\D$  whose
domains of linearity are precisely the
2-dimensional faces of $\SD$.
\item Given a regular subdivision $\SD$, consider the set of
all {\em concave} piecewise-linear functions on
$\D$ whose domains of linearity induce the
subdivision $\SD$.
As embedded in $\Real^{\D\cap\Z^2}$, this set is
a polyhedral cone. We take its image in the
quotient space
$\Real^{\D\cap\Z^2}/\Real\cdot(1,\dots,1)$ which
is again a polyhedral cone and  denote it  by
$\mc{T}C(\SD)$, called the {\em tropical cone} of
$\SD$. Its dimension   is called the {\em rank}
of $\SD$ and denoted by $\mr{rank}(\SD)$.
\end{enumerate}
\end{defns}
Now let $\psi:\D\cap\Z^2\rightarrow \Real$ be a
real-valued function defined on $\D\cap\Z^2$. We
construct a regular subdivision
$\D_{\psi}$ of $\D$ from $\psi$ as follows:\\
Let $G_{\psi}\subset\Real^3$ be the convex hull
of the set
\[\{(a, y): \quad y\le\psi(a),\quad a\in\D\cap\Z^2\}.\]
Then the upper boundary of $G_{\psi}$ is the
graph of a concave piecewise-linear function
$cc(\psi)$ which is called the \emph{concave
hull} of $\psi$. (The upper boundary of
$G_{\psi}$ is by definition the union of faces of
$G_{\psi}$ which do not contain vertical
half-lines.)
 Let $\D_{\psi}$ denote the regular
subdivision of $\D$ given by the domains of
linearity of $cc(\psi)$.\\
\begin{parbox}{11cm}
{The figure on the right illustrates  the case
when $\D$ is the segment, $conv(0,\dots,4)$. The
subdivision $\D_{\psi}$ is the union of two
segments $conv(0,2)$ and $conv(2,4)$ (We denote
the {\em convex hull} of a set $A$ by $conv
(A)$.)}
\end{parbox}\\
 Let us
consider the tropical cone
$\mc{T}C(\psi):=\mc{T}C(\D_{\psi})$.
 It contains the concave hull of $\psi$ restricted on $\D\cap\Z^2$:
 $cc(\psi)|_{\D\cap\Z^2}\in \mc{T}C(\psi)$. Notice that $cc(\psi)|_{\D\cap\Z^2}$ may not coincide with $\psi$
and in such case $\psi \notin \mc{T}C(\psi)$.
\begin{picture}(0,0)(-210,-100)
\curve(0,0,80,0) \put(-3,5){\tiny{$\bullet$}}
\put(16,8){\tiny{$\bullet$}}\put(36,28){\tiny{$\bullet$}}\put(56,17){\tiny{$\bullet$}}
\put(76,7){\tiny{$\bullet$}}
\put(0,10){$\vector(0,-1){30}$}\put(20,10){$\vector(0,-1){30}$}
\put(40,30){$\vector(0,-1){50}$}\put(60,20){$\vector(0,-1)
{40}$}\put(80,10){$\vector(0,-1){30}$}\put(40,-15){\tiny{${G_{\psi}}$}}
\linethickness{0.5mm}
\curve(0,10,40,30)\curve(40,30,80,10)
\end{picture}
\begin{defn}\label{defn:rank}
 The {\em  rank } of $\psi$, written as $\mr{rank}(\psi)$, is the rank of the regular subdivision $\D_{\psi}$.

\end{defn}
\begin{prop} Suppose that $\SD=\D_{\psi}$ is a regular nodal
subdivision of $\D$. Then,
\[\mr{rank}(\psi)=|\mr{\mr{Vertices}}(\SD)|-1-|\mr{Parallelograms}(\SD)|.\]
\end{prop}
The proof of this Proposition  can be easily
deduced from \cite[Lemma 2.40.]{IMS}.
\begin{rem}
The regular subdivisions of a lattice polytope is
studied in \cite[Ch.7]{GKZ} in which the
secondary fan is introduced. In a forthcoming
paper, the connection of the tropical cone
$\mc{T}C(\psi)$ and $\mr{rank}(\psi)$ to the
secondary fan will be studied.
\end{rem}
Now we study the adjacency graph $\SD^*$ of a
given subdivision $\SD$ of $\D$. By definition,
the vertices $F^*$ of $\SD^*$ correspond to the
2-dimensional faces $F$ of $\SD$ and two vertices
$F_1^*$ and $F_2^*$ of $\SD^*$ are connected by
an edge $(F_1^*, F_2^*)$ if the corresponding
faces $F_1$ and $F_2$ of $\SD$ have a common edge
$(F_1, F_2)$ in $\SD$.
Given an orientation $\Gamma$  on
$\mr{Edges}(\SD)$, we define an orientation
$\Gamma^*$ on $\mr{Edges}(\SD^*)$ as follows:
direct $F_1^*\rightarrow F_2^*$, if the oriented
edge $(F_1, F_2)$ and a normal vector to it
leaving from $F_1$ to $F_2$ are positively
oriented. Otherwise, direct $F_2^*\rightarrow
F_1^*$. For any subdivision $\SD$ of $\D$, we can
always find an orientation $\Gamma$ on
$\mr{Edges}(\SD)$ such that $\Gamma^*$ has
neither an oriented cycle nor a sink.  (A sink
(resp. source) is a vertex $v$ such that
 all edges adjacent to $v$ are coming into (resp. leaving from) $v$.) In fact,
 an oriented cycle in $\Gamma^*$ corresponds to a sink or a source at an internal vertex of $(\SD,\Gamma)$.
 Also, a sink in $\Gamma^*$ corresponds to an oriented cycle in $\Gamma$. We can choose an orientation $\Gamma$ on
 $\mr{Edges}(\SD)$ such that $\Gamma$ has no oriented cycle and it also has no sink/source at an internal vertex of $\SD$. For example, choose a generic vector
$\zeta\in\Real^2\setminus\{0\}$ and orient the
edges of $\SD$ so that they form acute angles
with $\zeta$.\\\\\\
\setlength{\unitlength}{0.22mm}
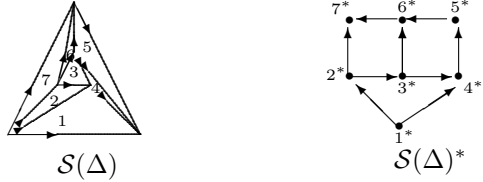
\begin{picture}(0,0)(-70,30)
\put(0,0){\vector(1,0){30}}
\curve(30,0,80,0)\put(30,5){\tiny{$1$}}
\put(25,17){\tiny{$2$}}
\put(0,0){\vector(1,2){15}}\curve(15,30,40,80)
\put(37,35){\tiny{$3$}}
\put(0,0){\vector(1,1){10}}\curve(10,10,30,30)
\put(50,25){\tiny{$4$}}
\put(0,0){\vector(2,1){10}}\curve(10,5,50,30)
\put(30,30){\vector(1,0){10}}\curve(40,30,50,30)
\put(45,50){\tiny{$5$}}
\put(30,30){\vector(1,2){10}}\curve(40,50,40,60)\put(35,45){\tiny{$6$}}
\put(30,30){\vector(1,4){5}}\curve(35,50,40,80)\put(20,30){\tiny{$7$}}
\put(40,50){\vector(0,1){10}}\curve(40,60,40,80)
\put(40,50){\vector(1,-2){5}}\curve(45,40,50,30)
 \put(40,50){\vector(3,-4){9}}\curve(49,38,80,0)
\put(50,30){\vector(1,-1){10}}\curve(60,20,80,0)
\put(40,80){\vector(1,-2){10}}\curve(50,60,80,0)
\put(30,-23){$\SD$}
\end{picture}
\begin{picture}(0,0)(-270,-5)
\put(0,0){\vector(1,0){30}}
\put(33,0){\vector(1,0){30}}\put(0,0){\vector(0,1){30}}
\put(33,0){\vector(0,1){30}}\put(33,0){\vector(0,1){30}}
\put(30,-30){\vector(-1,1){25}}\put(30,-30){\vector(3,2){35}}

\put(30,35){\vector(-1,0){25}}\put(60,35){\vector(-1,0){25}}
\put(67,0){\vector(0,1){30}}
\put(-2,-2){\tiny{$\bullet$}}\put(-13,-2){\tiny{$2^*$}}\put(30,-2)
{\tiny{$\bullet$}}\put(30,-10){\tiny{$3^*$}}
\put(64,-2){\tiny{$\bullet$}}\put(70,-10){\tiny{$4^*$}}
\put(-2,32){\tiny{$\bullet$}}\put(-10,38){\tiny{$7^*$}}\put(30,32){\tiny{$\bullet$}}
\put(30,39){\tiny{$6^*$}}\put(62,32){\tiny{$\bullet$}}\put(62,39){\tiny{$5^*$}}
\put(28,-32){\tiny{$\bullet$}}
\put(28,-40){\tiny{$1^*$}}\put(28, -55){$\SD^*$}
\end{picture}


\subsubsection{Shustin's characterizations}
\label{sec:nodal curves}Let
 $V_{\K}(f)$  be a curve defined by a
 Laurent polynomial over $\K$,
\[f=\sum_{a\in\D\cap\Z^2} c_a(t)x^a\in\K[\Z^2],\qquad c_a\in\mb{K}
\setminus\{0\},\] where $\D=\mr{Newton}(f)$.
The support of the
tropicalization
 of $V_{\K}(f)$ is  the corner locus of the piecewise-linear  function
\[\Real^2\rightarrow \Real, \quad\alpha\mapsto max_{a\in\D\cap\Z^2}\{a\cdot\alpha +
Val_f(a)\},\] where  $Val_f$ is the function
\[Val_f:\D\cap\Z^2\rightarrow \Z, \quad a\mapsto
 Val(c_a(t)).\]
Let us denote the support of the tropicalization
of the curve $V_{\K}(f)$ by $\tau_f$.
In general,  the tropical curve $\tau_{\w}$ is by
definition the corner locus of the
piecewise-linear  function
\[\Real^2\rightarrow \Real, \quad\alpha\mapsto max_{a\in\D\cap\Z^2}\{a\cdot\alpha +
\w(a)\},\] where $\w:\D\cap\Z^2\rightarrow\Z$ is
an integral-valued
function on $\D\cap\Z^2$.
The regular subdivision $\D_{\w}$ of $\D$ is dual
to the tropical curve  $\tau_{\w}$ in the
following sense (\cite[\S 2.5.1]{IMS}):
\begin{itemize}\item the components of $\Real^2\setminus
 \tau_{\w}$ are in  1-to-1 correspondence with $\mr{Vertices}(\D_{\w})$;
 \item the edges of $\tau_{\w}$ are in 1-to-1 correspondence with
 $\mr{Edges}(\D_{\w})$ so that an edge $e$ of $\tau_{\w}$ is dual to an
  edge of $\D_{\w}$ which is orthogonal to $e$ with the lattice length
  $\w(e)$;
 \item the vertices of $\tau_{\w}$ are in 1-to-1 correspondence
 with the 2-dimensional faces of $\D_{\w}$
  so that the valency of a vertex of $\tau_{\w}$ is
 equal to the number of sides of the dual face.
\end{itemize}
We call a tropical curve $\tau_{\w}$ to be simple
(respectively, triangular, nodal) if the dual
regular subdivision $\D_{\w}$ of $\D$ is simple
(respectively, triangular, nodal). The {\em rank}
of $\tau_{\w}$ is by definition the rank of
$\D_{\w}$.

\begin{eg}\label{eg:11} Let $f=(1+x+y)(1+tx+ty)$.  The picture below on the right is $\tau_f$ and one on the
left is the corresponding  subdivision of the
Newton polygon of $f$.\end{eg}

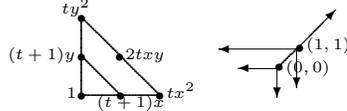
\setlength{\unitlength}{0.13mm}
\begin{picture}(0, 0) (-200,-55)
\curve(0,0,0,-80)\curve(0,-80,80,-80)
\curve(80,-80,0,0)\curve(0,-40,40,-80)
\put(-20,3){\tiny{$ty^2$}}
\put(-75,-43){\tiny{$(t+1)y$}}\put(-15,
-83){\tiny{$1$}}\put(16,-92){\tiny{$(t+1)x$}}
\put(87,-83){\tiny{$tx^2$}}\put(45,-43){\tiny{$2txy$}}
\put(-5,-6){\tiny{$\bullet$}}\put(-5,-45){\tiny{$\bullet$}}\put(-5,-85)
{\tiny{$\bullet$}}\put(35,-85){\tiny{$\bullet$}}
\put(75,-85){\tiny{$\bullet$}}\put(34,-45){\tiny{$\bullet$}}
\put(200,-53){\vector(-1,0){40}}
\put(200,-53){\vector(0,-1){30}}\put(209,-57){\tiny{$(0,0)$}}\put(197,-56){\tiny{$\bullet$}}
\curve(200,-53,220,-33)\put(218,-36){\tiny{$\bullet$}}
\put(230,-33){\tiny{$(1,1)$}}
\put(220,-33){\vector(-1,0){80}}\put(220,-33){\vector(1,1){40}}\put(220,-33){\vector(0,-1){40}}
\end{picture}\\\\
Now let us study the initial schemes of the curve
$V_{\K}(f)$, which are defined by the initial
polynomials of $f$. For this,  we rewrite
the coefficients of $f$,\\
$c_a(t)=\bar{c_a}t^{Val_f(a)}+l.o.t.,\quad
\bar{c_a}\in (\mb{C})^*,\quad a\in \D\cap\Z^2$ as
follows:
\begin{equation}c_a(t)=c_a^{\circ}\text{} t^{\nu_f(a)}+l.o.t.,\end{equation}\label{eq:11}
where $\nu_f$ is the concave hull of $Val_f$
(defined as in \ref{sec:subdivision}) and
$c_a^{\circ}$ is some complex number which is
zero if $\nu_f(a)>Val_f(a)$. Given  a point $\al$
which is either a vertex or a point in the
relative interior of an edge of $\tau_f$, the
maximum of $a\cdot\al+\nu_f(a)
\quad(a\in\D\cap\Z^2)$ is attained on the
corresponding dual face or edge in
$\D_{f}:=\D_{\nu_f}$. Therefore,
 \[\n\al f=\sum c_a^{\circ} x^a,\] where the sum
runs over the lattice points on the dual face or edge in $\D_{f}$.
 In particular, we consider the initial polynomials $\n{\al_i} f$,
 where $\al_i$ are the vertices  of $\tau_f$
 corresponding to the  faces
(2-dimensional subpolygons)  $\D_i$ in
$\D_f:\D_1\cup\dots\cup\D_m$.
 Then,  given  a curve $V_{\K}(f)$ in $\TK^2$ with $\mr{Newton}(f)=\D$ we
obtain a collection of complex curves
$V(\n{\al_i} f)$ in $\TC^2$ with
$\mr{Newton}(\n{\al_i} f)=\D_i ,\quad (i=1,\dots,m)$.
This collection of complex polynomials together
with the subdivision $\D_f$ completely determines
the tropicalization of the curve $V_{\K}(f)$ (or
of $f$).

Now we can present Shustin's characterizations of
the tropicalizations of nodal curves with a given
rank condition:
\begin{thm}\cite[\S 3.3]{Shustin2}\label{thm:Shustin's characterization} Let $V_{\K}(f)$ be a curve with $\delta$ nodal
singular points (i.e. ordinary double points) as
the only singularities, where $\delta$ is a
natural number with $\delta\le
|\mr{\mr{Int}}(\D)\cap\Z^2|$. Suppose that $\mr{rank}(\D_f)\ge
r$, where $r= |\D\cap\Z^2|-1-\delta$. Then the
following holds true:
\begin{enumerate}\item(combinatorial) the regular subdivision $\D_f:\D_1\cup\dots\cup\D_m$ is  simple, nodal,
and $\mr{rank}(\D_{f})= r$;
 \item(geometric)\begin{itemize}\item
for each triangle $\D_i$, the curve
$V(\n{\al_i}f)$ is rational and meets the union
of toric divisors $Tor(\partial\D_i)$ at exactly
three points, where it is unibranch;
\item for each
parallelogram $\D_j$, the polynomial $ \n{\al_j}f$
has the form \[x^ky^l(\alpha x^a + \beta
y^b)^p(\gamma x^c +\delta y^d)^q\] with $(a, b) =
(c, d) = 1, (a:b)\neq (c:d), \alpha,\beta,\gamma,\delta\in\C\setminus\{0\}$
\end{itemize}
\end{enumerate}
\end{thm}
Shustin's original statement(\cite[\S
3.3]{Shustin2}) has the following base-point
condition instead of the rank condition,
$\mr{rank}(\D_f)\ge r$: Suppose that the curve
$V_{\K}(f)$ passes through $r$ generic points
$p_1,\dots,p_r\in \TK^2$  such that
$q_1=Val(p_1),\dots, q_r=Val(p_r)\in\Real^2$
(taking valuations coordinatewise) are generic
points (Definition \ref{defn:GeneralPosition}),
where $r=|\D\cap\Z^2|-1-\delta$.

 We
explain why we can replace the base-point
condition by the rank condition:  His proof is by
estimating $\check{\chi}(C^{(t)})$, the
topological Euler characteristic of the
normalization of the complex curve
$C^{(t)}=V(f_t)$ (see Proposition
\ref{prop:Equisingular}) from above and from
below and comparing the bounds.
The upper bound of $\check{\chi}(C^{(t)})$ in his
computation holds true for any flat deformation.
For a lower bound, he uses the following
inequality:
\[\begin{array}{ccl}
\check{\chi}(C^{(t)}) &=&2-2g(C^{(t)})\\ &=& 2-2
(|\mr{Int}(\D)\cap\Z^2|-\delta)\\
&=&2-2|\mr{Int}(\D)\cap\Z^2|+
2(|\D\cap\Z^2|-1-r)\\&=&2|\partial\D\cap\Z^2|-2r\\
&\ge& 2|\partial\D\cap \Z^2|-2 \mr{rank}(\D_f)
\end{array}\]
In the second equality, he used the fact that
$C^{(t)}$ has $\delta$ nodal points. The last
inequality follows from the condition, $r\le
\mr{rank} (\D_f)$. In fact, Shustin called the
points $q_1,\dots, g_r$ to be generic when $r\le
\mr{rank} (\D_f)$. (Definition
\ref{defn:GeneralPosition})

\subsection{Tropical Intersection Theory}\label{sec:IT}
\cite{AR}, \cite{Katz2}, \cite{Kazarnovskii2},
\cite{Kazarnovskii3}

Tropical cycles  form a graded commutative
algebra $\mf{A}$.
To a subvariety $X$ of the algebraic torus
$\T_{\K}=\mr{Spec}(\K[\Z^n])$, we can assign an
element $\Tr(X)$ of $\mf{A}$, namely the
tropicalization of $X$. This correspondence
determines an intersection theory of subvarieties
of $\T_{\K}$.
 In this paper,
we only summarize about the product
$\mc{T}_1\cdot\mc{T}_2$ when $\mc{T}_1$ and
$\mc{T}_2$ are  complementary dimensional
tropical varieties in $\Real^n$. The support of
the product $\mc{T}_1\cdot\mc{T}_2$ is by
definition the zero-dimensional strata of the set
$supp(\mc{T}_1)\cap supp(\mc{T}_2)$. It is a
finite set of points in $\Real^n$. The weighting
function $\bsm=\bsm_{\mc{T}_1\cdot\mc{T}_2}$ in
the product $\mc{T}_1\cdot\mc{T}_2$ is defined as
follows:
let $\w\in supp(\mc{T}_1)\cap supp(\mc{T}_2)$.
There are two possible
cases that  $\w$ is the intersection point of a transversal intersection or not.

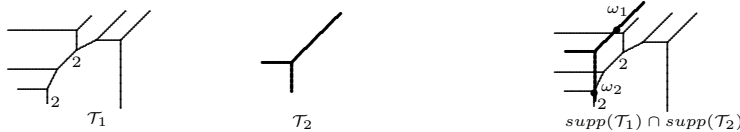
\setlength{\unitlength}{0.13mm}
\begin{picture}(110,0)(-200,40) 
\curve(0,0,-30,0)\curve(0,0,0,-15)\put(3,-18){\tiny{$2$}}\curve(0,0,10,20)
\curve(10,20,-40,20)\curve(10,20,30,40)\put(25,23){\tiny{$2$}}
\curve(30,40,30,60)\curve(30,40,50,50)\curve(30,60,-40,60)
\curve(30,60,45,75)\curve(50,50,75,50)\curve(50,50,80,80)
\curve(75,50,105,80)\curve(75,50,75,-20)
\put(40,-35){\tiny{$\mc{T}_1$}}
\end{picture}

\begin{picture}(110,0) (-450,-20) 
\linethickness{0.3mm}
 \curve(0,0,-30,0)
 \curve(0,0,0,-30)
 \curve(0,0,50,50)
 \put(0,-65){\tiny{$\mc{T}_2$}}
\end{picture}
\begin{picture}(310,0)(-640,10) 
\curve(0,0,-30,0)\curve(0,0,0,-15)\put(3,-18){\tiny{$2$}}
\curve(0,0,10,20)
\curve(10,20,-40,20)\curve(10,20,30,40)\put(25,23){\tiny{$2$}}
\curve(30,40,30,60)\curve(30,40,50,50)\curve(30,60,-40,60)
\curve(30,60,45,75)\curve(50,50,75,50)\curve(50,50,80,80)
\curve(75,50,105,80)\curve(75,50,75,-20)
\put(-30,-35){\tiny{$supp(\mc{T}_1)\cap
supp(\mc{T}_2)$}}
\put(-5,-6){\tiny{$\bullet$}}\put(7,0){\tiny{$\w_2$}}
\end{picture}
\begin{picture}(130,0)(-322,-31) 
\linethickness{0.3mm}
 \curve(0,0,-30,0)
 \curve(0,0,0,-50)
 \curve(0,0,50,50)
 \put(17,18){\tiny{$\bullet$}}\put(12,39){\tiny{$\w_1$}}
\end{picture}\\\\\\
In the first case, $\w$ is a regular point of
each $\mc{T}_i$ and so $\mc{T}_i$ is equal to
$L_i$ locally near $\w, i=1,2$, where $L_1$ and
$L_2$ are affine spaces of complementary
dimensional.
\begin{enumerate}\item \label{extrinsic}The \emph{extrinsic} intersection
multiplicity of  $\mc{T}_1$ and $\mc{T}_2$ at
$\w$, denoted by $\xi(\w;\mc{T}_1,\mc{T}_2)$, is
the volume of the parallelepiped constructed by
the fundamental cells of the lattices
$\mb{L}_i\cap\Z^n,(i=1,2)$ (``principal
parallelepiped'')
\item \label{tropical}The \emph{tropical} intersection multiplicity of $\mc{T}_1$
and $\mc{T}_2$ at $\w$ is
\[\bsm(\w)=\bsm(\w;\mc{T}_1,\mc{T}_2):= \bsm_{\mc{T}_1}(\w)\cdot\bsm_{\mc{T}_2}(\w)\cdot
\xi(\w;\mc{T}_1,\mc{T}_2)\]
\end{enumerate}
In the second case, $\w$ is not a regular point
of either $\mc{T}_1$ or $\mc{T}_2$. However, by a
small local displacement of $\mc{T}_1$ and
$\mc{T}_2$, we can achieve the transversality
near $\w$. (The details can be found in
\cite{FS}, \cite{MikhalkinICM}). Then $\bsm(\w)$
is by definition the sum of all
$\bsm({\tilde{\w}})$, where $\tilde{\w}$'s are
the transversal intersection points appearing in
the displacement.

When $\mc{T}_i=\Tr(X_i)$ for subvarieties $X_i$
of $\T_{\K} (i=1,2)$, the sum of weights on
the product $\mc{T}_1\cdot\mc{T}_2$ is equal to
the number of intersection points in $X_1\cap
gX_2$ for a generic $g\in\T_{\K}$. It is called
the {\em degree} of the product
$\mc{T}_1\cdot\mc{T}_2$ and denoted by
\[(\Tr(X_1)\cdot\Tr(X_2)).\]
For example,

\begin{itemize}

\item the degree of the product of  tropicalizations of $n$ hypersurfaces in $\T_{\C}$
is  the mixed volume of the Newton polyhedra of
the hypersurfaces times $n!$; (compare
\cite{Bernstein}.)
\item let $\mb{T}_1$ and $\mb{T}_2$ be two subtori of $\T_{\C}$ of complementary
dimension. Then their tropicalizations are
rational linear subspaces of  $\Real^n$ with
constant multiplicity $1$ and they are of
complementary dimension.  The degree of the
product $\Tr(\mb{T}_1)\cdot\Tr(\mb{T}_2)$ is
equal to the (normalized) volume of the
parallelepiped defined by the fundamental cells
of the lattices,  $\Tr(\mb{T}_1)\cap\Z^n$ and
$\Tr(\mb{T}_2)\cap\Z^n.$
\end{itemize}

\section{Intermediate parameter spaces}\label{sec:intermediate} In this section, we study certain parameter spaces of curves on toric surfaces, which
 are closely related to the initial
schemes of Severi varieties. They are defined by
considering Shustin's characterization (Theorem
\ref{thm:Shustin's characterization}) in the view
of parameter spaces and turn out to be very
simple.

Let us consider    the projective toric surface
$X_{\D}$ constructed from a 2-dimensional lattice
polygon $\D$ in $\Real^2$. That is,
$X_{\D}\subset\Pro^{n-1}=\Pro(\C^n)$ is the
closure of the set
\[X_{\D}^{\circ}=\{(x^{a_1}:\dots:x^{a_n}):x=(x_1,x_2)\in\TC^2\},\] where $\D\cap\Z^2=\{a_1,
\dots,a_n\}$.
We can identify the projectivization of the dual
space of $\C^n$, $\Pro((\C^n)^*)$,  as the
projectivization of the vector space of Laurent
polynomials whose Newton polygons are subsets of
the polygon $\D$, which is called the
tautological linear system of curves on the toric
surface $X_{\D}$ and denoted by $\Pro_{\D}$. We
study several subvarieties $\mb{V}_{\bullet}$ of
this linear system $\Pro_{\D}$. The study of
$\mb{V}_{\bullet}$ was motivated by trying to
understand the initial schemes of Severi
varieties. However, besides the roles as building
blocks to understand the Severi varieties, the
author believes that the $\mb{V}_{\bullet}$s have
their own independent interests and also they may
be generalized in many different
perspectives. Let $\T_{\D}$ be the big open torus of $\Pro_{\D}$.
%
%
\subsection{$\mb{V}_{\partial\mc{S}(\D),\Pro^1}$}
\label{sec:SD,rational}
\begin{defns}
 Let $\mb{V}_{\partial\D}$ denote the set of all $f\in\mb{T}_{\D}$ such that the restriction of $f$ on each
 edge of $\D$ is a pure power of a binomial (up to multiplication by a
 monomial), i.e. of the form of $x^ay^b(\alpha x^c+\beta y^d)^s$,
 where $a,b,c,d\in\Z, \alpha, \beta\in\mb{C}^*$ and $s$ is the lattice
 length of the edge.
 Geometrically, points  of $\mb{V}_{\partial\D}$ correspond to curves on the toric surface $X_{\D}$ such that
 they cross the union of the toric divisors
 at precisely $l$ points, where $l$ is the number of edges of $\D$. More generally,  we consider subdivisions $\SD$ of the polygon
$\D$. Let $\mb{V}_{\partial\mc{S}(\D)}$ denote the set
of all $f\in\T_{\D}$ such that
$f_{\D_i}\in\mb{V}_{\partial\D_i}$ for every
$\D_i\in \mr{Faces}(\SD)$, where $f_{\D_i}$ is the
 restriction of $f$ on $\D_i$. Let $\mb{V}_{\partial\mc{S}(\D),\Pro^1}$ denote the set of all $f\in\mb{V}_{\partial\mc{S}(\D)}$
 such that $f_{\D_i}$ defines a rational curve which is unibranch at each intersection point with the boundary divisors of the toric surface $X_{\D_i}$ for every
$\D_i\in \mr{Faces}(\SD)$.
\end{defns}

\begin{lem}\label{lem:parameterization}\cite[Lemma 3.5.]{Shustin2} If $\D=\blacktriangle$ is a triangle,
every
$f\in\mb{V}_{\partial\blacktriangle,\Pro^1}$ can
be given by the following rational
parametrization,
\[\ta\mapsto(\al\ta^{s_1v_{11}}(\ta+1)^{s_2v_{21}}, \be\ta^{s_1v_{12}}(\ta+1)^{s_2v_{22}}),\]
where $\al,\be\in\TC^2$ and $v_1=(v_{11},v_{12}),
v_2=(v_{21},v_{22})$ are two vectors among the
three primitive inner-normal vectors to the edges
of the triangle $\blacktriangle$, and $s_1$ and
$s_2$ are the lattice lengths of the
corresponding edges of $\blacktriangle$.
\end{lem}

\begin{thm}\label{thm:SD,rational}
 Suppose that $\mc{S}(\D)$ is triangular, that is, every face $\D_i$ of $\SD$ is a triangle. Then the following hold true.
\begin{enumerate}
 \item $\mb{V}_{\partial\mc{S}(\D),\Pro^1}$ is a translation of a
 closed subgroup of the torus $\T_{\D}$.
 \item Its dimension is equal to $|\mr{Vertices}(\mc{S}(\D))|-1$.

\end{enumerate}
\end{thm}

\begin{proof} There are three steps to complete the proof.
 First, we show that $\mb{V}:=\mb{V}_{\partial\mc{S}(\D),\Pro^1}$ is not empty.
Second, we construct a closed subgroup $\mb{G}$
of $\T_{\D}$ with dimension
$|\mr{Vertices}(\mc{S}(\D))|-1$.
Last, we show that $\mb{V}$ is equal to the
translation $f\cdot\mb{G} $ of any point
$f\in\mb{V}$.

{\em  Step 1.} From the Lemma
\ref{lem:parameterization}, we know that for a triangle
$\D$ any element in $\mb{V}_{\partial\D,\Pro^1}$
is uniquely determined by an element
$(\al,\be)\in\TC^2$ by assuming  that one of  the vertices of $\D$ is the
  origin and the constant term of
an equation is always 1. Let us denote this
element by $f^{(\al,\be)}$.  We extend this
argument to the many-triangles case,
$\mc{S}(\D):\D_1\cup\cdots\cup\D_m$. We know that
we can always find a $(\al,\be)\in\TC^2$ such
that $f^{(\al,\be)}$ (up to the multiplication by
a monomial) satisfies a given prescription on any
two of the three edges of a triangle $\D$. That
is, the following data are prescribed: the
coefficients of the equation $f^{(\al,\be)}$ at
the vertices of $\D$  and the intersection points
of the rational curve (defined by
$f^{(\al,\be)}$) with toric divisors
corresponding to two of the three edges of $\D$.
(for a proof, see \cite[Lemma 3.5.]{Shustin2})
Now we choose an orientation  on the adjacency
graph $\SD^*$ of $\SD$ which has no oriented
cycle and no sink at vertices of 3-valency. (see
\S \ref{sec:subdivision} for details.)
It is clear that such oriented adjacency graph
provides an algorithm to construct a point in
$\mb{V}$. That is, we can choose a consistent
collection of $(\al,\be)_{\D_i}$ for sub-polygons
 $\D_i$ in $\SD$. Therefore $\mb{V}$ is not empty.

{\em  Step 2.} Choose a linear order,
$\D_1,\dots,\D_m$, in the set of triangles in the
subdivision $\mc{S}(\D)$. We choose one of the
vertices of $\D_1$ and assume that it is the
origin.  We order the set of inner edges in
$\mc{S}(\D)$ in the following way: Choose all
inner edges belonging to $\D_1$(there are at most
three such edges). Put them in an order. Then,
choose all inner edges belonging to $\D_2$ except
the ones which may belong to $\D_1$. Add them in
an order to the first set. In this way, we put a
linear order in the set of all inner edges.
Each inner edge defines two binomial equations as
follows: Let $l=s_{ij}$ be the inner edge shared
by $\D_i$ and $\D_j$, $i<j$. Let $a=(a_1,a_2)$ be
the lattice point of one of the two ends of $l$
and let $v=(v_1,v_2)$ be the primitive vector
along $l$ from $a$.
\begin{equation}
 \gamma_i\al_i^{a_1}\be_i^{a_2}-\gamma_j\al_j^{a_1}\be_j^{a_2}=0;\end{equation}

 \begin{equation}
 \al_i^{v_1}\be_i^{v_2}-\al_j^{v_1}\be_j^{v_2}=0
\end{equation}

\setlength{\unitlength}{0.3mm}
\begin{picture}(0,0)(-40,-30)
 \curve(0,0,80,0)\curve(0,0,10,30)\curve(10,30,80,0)
\curve(0,0,20,-20)\curve(20,-20,80,0)\put(-2,-2){\tiny{$\bullet$}}
\put(-5,4){$a$}\put(30,10){$\D_i$}\put(30,-10){$\D_j$}\put(15,5){$v$}\vector(1,0){20}
\end{picture}

We collect the binomials for all inner edges and
add one more
binomial, $\gamma_1=1$. Let us denote this system by $(\bigstar)$.
Then it is clear that this system is uniquely
determined by the following system,
\[\gamma_1=1 ;\]
\[\al_i^{v_1}\be_i^{v_2}\al_j^{-v_1}\be_j^{-v_2}=1(\bigstar\bigstar)\label{binomials}\] where
the monomials in the left hand side of the equations are collected for all the inner edges.
Let $\mb{G}$ be the closed subgroup in the torus $\TC^3\times\cdots\times\TC^3$ with coordinates
$(\bs{\al},\bs{\be},\bs{\gamma})=((\al_1,\be_1,\gamma_1),\dots,
(\al_m,\be_m,\gamma_m))$ defined by the system
$(\bigstar)$.
Let $M:=M_{\partial\mc{S}(\D),\Pro^1}$ be the
matrix corresponding to the monomials in the left
hand side of equations in ($\bigstar\bigstar$)
where the rows are indexed by the inner edges and
the columns are indexed by  $(\al_i, \be_i),
i=1,\dots,m$. It is straightforward to see that
the rows of $M$ are linearly independent.
Therefore,
\[\mr{dim}(\mb{G})=2|\mr{\mr{Triangles}}
(\mc{S}(\D))|-|\mr{IEdges}(\mc{S}(\D))|,\] where
$\mr{IEdges}(\mc{S}(\D))$ is the set of all inner
edges in $\SD$.  Also, by the  Lemma
\ref{lem:Vertices-1},
\[\mr{dim}(\mb{G})=|\mr{Vertices}(\mc{S}(\D))|-1.\] Now we embed
$\mb{G}$  into $\T_{\D}$ in the following way,
\[ \Phi=(\Phi_{(w_1,w_2)})_{(w_1,w_2)\in\D\cap\Z^2}: \mb{G}\rightarrow \T_{\D},\]
\[ \Phi_{(w_1,w_2)}((\al_1,\be_1,\gamma_1),\dots,(\al_m,\be_m,\gamma_m))=\gamma_k\al_k^{w_1}\be_k^{w_2},\]
where $(w_1,w_2)\in\D_k\cap\Z^2, k=1,\dots,m$.
This map is well-defined because $\mb{G}$
satisfies the system $(\bigstar)$. Also this map is injective by the Lemma \ref{lem:injective}.

{\em Step 3.} Let us first show that
$f\cdot\mb{G} \subset\mb{V}$ for any $f\in
\mb{V}$. (We don't distinguish $\mb{G}$ from its
image under the embedding $\Phi$.) Let
$(\bs{\al},\bs{\be},\bs{\gamma})\in\mb{G}$. The
restriction of
$f\cdot(\bs{\al},\bs{\be},\bs{\gamma}) $ on
$\D_k$ is given by $\gamma_kf_{\D_k}(\al_kx,
\be_k y)$, which is a point in $\mb{V}_{\D_k,
\Pro^1}  (k=1,\dots, m)$. Thus, it is enough to
show that the restrictions of
$f\cdot(\bs{\al},\bs{\be},\bs{\gamma}) $ on all
sub-triangles coincide along the inner edges. It
follows from
the fact that $\mb{G}$ satisfies the system $(\bigstar)$.
Now we show that the other inclusion also holds.
Let $h,h'\in\mb{V}$. Then the restriction of $h$
(resp. $h'$) on $\D_k$ has the following form up
to the multiplication  by a monomial
$x^{b_1}y^{b_2}$,
\[h_{\D_k}(x,y)=\gamma_k(f^{(\al_k,\be_k)})=\gamma_k f^{(\al_k,\be_k)}(\al_kx,\be_ky)\]
\[h'_{\D_k}(x,y)=\gamma_k'(f^{(\al_k',\be_k')})=\gamma_k'f^{(\al_k',\be_k')}(\al_k'x,\be_k'y)
,\] for some $(\al_k,\be_k,\gamma_k)$ (resp.
$(\al_k',\be_k',\gamma_k')\in\TC^3$), $(b_1,b_2)\in\D\cap\Z^2
$, where
$k=1,\dots,m$.
Thus
$h'_{\D_k}(x,y)=\gamma_k'\gamma_k^{-1}h_{\D_k}(\al_k^{-1}\al_k'x,
\be_k^{-1}\be_k'y)$.
That is, $h'_{\D_k}$ is the restriction  of
$(\bs{\al}^{-1}\bs{\al}',\bs{\be}^{-1}\bs{\be}',\bs{\gamma}'\bs{\gamma}^{-1})\cdot
h$ on $\D_k$. Therefore \[h'=
(\bs{\al}^{-1}\bs{\al}',\bs{\be}^{-1}\bs{\be}',\bs{\gamma}'\bs{\gamma}^{-1})\cdot
h.\]
We have completed the proof.
\end{proof}
\begin{rem}
We can compute the number of components
of $\mb{V}_{\partial\SD,\Pro^1}$  easily from the
matrix $M_{\partial\SD, \Pro^1}$.  It is equal to
the greatest common divisor of all the absolute
values of $l\times\l$ minors of $M_{\partial\SD,
\Pro^1}$. Also it is equal to the number of
lattice points in the parallelepiped
$P=\{x_1v_1+\cdots+x_lv_l: 0\le x_i <1,
i=1,\dots,l\}$, where $v_1,\dots, v_l$ are the
row vectors of $M_{\partial\SD, \Pro^1}$.
\end{rem}

\begin{lem}\label{lem:Vertices-1} If a subdivision $\SD$ of $\D$  is
triangular, then
 \[2|\mr{\mr{Triangles}}(\mc{S}(\D))|-|\mr{IEdges}(\mc{S}(\D))|=|\mr{Vertices}(\SD)|-1\]
\end{lem}
\begin{proof}
 Let $F$ be the number of the triangles, $ E'$ be the number of the
edges on the boundary of $\D$, and let  $IE$ be
the number of the inner edges in $\mc{S}(\D)$,
respectively. Then, $3\cdot F=E'+2\cdot IE$.
Since
$|\mr{Vertices}|-|\mr{Edges}|+|\mr{\mr{Faces}}|=1$,\[\begin{array}{rcl}

|\mr{Vertices}(\SD)|-1& =& |\mr{Edges}|-|\mr{\mr{Faces}}| \\&=& (E' +
IE)-F \\&=& (3\cdot F-IE)-F\\&=&2\cdot F-IE.
\end{array}\]
\end{proof}

\begin{lem}\label{lem:injective} The map $\Phi=(\Phi_{(w_1,w_2)})_{(w_1,w_2)\in\D\cap\Z^2}$ defined in the proof of Theorem \ref{thm:SD,rational} is injective.
\end{lem}

\begin{proof} Given a vertex $v_1$  of
  any triangle $\D_k, (k=1,\dots,m)$, we can find two lattice points $v_2, v_3$ on $\D_k$ such that the convex hull of $v_1,v_2,v_3$ is a primitive triangle, that is, it has no interior lattice point. Now it is clear that any given values for $\Phi_{v_i} (i=1,2,3)$ uniquely determine $\al_k,\be_k,\gamma_k$.
\end{proof}

\begin{eg} Let $\SD$ be the following subdivision with 3
2-dimensional faces $F_1,F_2,F_3$.
Since $\SD$ has no interior lattice point,
$\mb{V}_{\partial\SD}=\mb{V}_{\partial\SD,\Pro^1}$.
\setlength{\unitlength}{0.25mm}
\begin{picture}(0,0)(10,95)
\curve(0,0,0,80)\curve(0,0,20,40)\curve(0,0,40,40)
\curve(20,40,40,40)\curve(20,40,0,80)\curve(40,40,0,80)
\put(-2,-2){\tiny{$\bullet$}}\put(-2,18){\tiny{$\bullet$}}\put(-2,38){\tiny{$\bullet$}}
\put(-2,58){\tiny{$\bullet$}}\put(-2,78){\tiny{$\bullet$}}\put(18,38){\tiny{$\bullet$}}
\put(18,58){\tiny{$\bullet$}}\put(18,18){\tiny{$\bullet$}}\put(38,38){\tiny{$\bullet$}}
\put(20, 30){\tiny{$F_1$}}\put(5,
43){\tiny{$F_2$}}\put(25, 45){\tiny{$F_3$}}
\end{picture}%
We get the matrix
$M_{\partial\SD,\Pro^1}$:
\[M_{\partial\D,\Pro^1}=\left(\begin{array}{ccccccc}
&\al_1&\be_1&\al_2&\be_2&\al_3&\be_3\\
s_{12}&1&2&-1&-2&0&0\\
s_{13}&1&0&0&0&-1&0\\
s_{23}&0&0&-1&2&1&-2
\end{array}\right)\]
The Smith Normal Form of
$M_{\partial\SD,\Pro^1}$ is
:\[\left(\begin{array}{cccccc}
1&0&0&0&0&0\\0&1&0&0&0&0\\0&0&2&0&0&0
\end{array}\right).\]
Thus
$\mb{V}_{\partial\SD}=\mb{V}_{\partial\SD,\Pro^1}$
is a union of two translations of 3-dimensional
subtorus of $\T_{\D}$.
\end{eg}

 \subsection{$\mb{V}_{\partial\mc{S}(\D),nodal}$}\label{sec:SD,nodal}
Now we allow to have parallelograms in a subdivision $\SD$.
\begin{defns} For a parallelogram $\blacksquare$, let $\mb{V}_{\partial\blacksquare,\Pro^1+\Pro^1}$
denote the set of all $f\in\Pro_{\D}$ such that
$f$ is the product of Laurent polynomials
$f_1, f_2$  whose Newton polygons are
two nonparallel sides of $\blacksquare$ and each $f_i$ is a
pure power of a binomial for  $i=1, 2$. Let $\mc{A}^c\subset\blacksquare\cap\Z^2$ be the set of
lattice points in $\blacksquare$ which are not in the
lattice  generated  by the primitive vectors
 along the sides  of $\blacksquare$. Any element of $\mc{A}^c$ is called
\emph{special}. If $\mc{A}^c$ is not empty, then
$\blacksquare$ is called \emph{non-primitive}.\\
\begin{parbox}
{10.5cm}{ (The figure on the right shows that
there are 5 lattice points in $\blacksquare$. The
unique interior lattice point is special.)}
\end{parbox}
\setlength{\unitlength}{0.23mm}
\begin{picture}(200,0)(-10, 5)
 \curve(0,0,20,0)\curve(0,0,20,40)\curve(20,40,40,40)\curve(40,40,20,0)
\put(-2,-2){\tiny{$\bullet$}}\put(18,-2){\tiny{$\bullet$}}\put(17,37){\tiny{$\bullet$}}\put(37,37){\tiny{$\bullet$}}
\put(17,18){$\circ$}
\end{picture}\\
Notice that  $f$ has no term
corresponding to the monomial $x^a$ for any
$f\in \mb{V}_{\partial\blacksquare,\Pro^1+\Pro^1}$ and
$a\in\mc{A}^c$. Therefore
$\mb{V}_{\partial\blacksquare,\Pro^1+\Pro^1}$ is
contained in the coordinate subspace  of
$\Pro_{\blacksquare}$ defined by the linear equations
$z_a=0, a\in\mc{A}^c$. If $\mc{A}^c$ is
empty, then the coordinate space is  the ambient
space $\Pro_{\blacksquare}$ itself. Let $\T_{\mc{A}}$ be
the big open torus of this coordinate subspace.

Now suppose that
$\mc{S}(\D):=\D_1\cup\cdots\cup\D_m$ is nodal,
that is, every sub-polygon is either a triangle
or a parallelogram. Let $\mb{V}_{\partial\mc{S}(\D),nodal}$ denote the set of all
  $f\in\Pro_{\D}$  with the following properties,
\begin{itemize}
 \item ($\blacktriangle$) For every triangle $\D_i$, $f_{\D_i}\in\mb{V}_{\partial\D_i,\Pro^1}$;
 \item ($\blacksquare$) For every parallelogram $\D_j$, $f_{\D_j}\in\mb{V}_{\partial\D_j,\Pro^1+\Pro^1}$
\end{itemize}
\end{defns}
We
consider the subset  $\mc{A}=\mc{A}_1\cup\cdots\cup\mc{A}_m$ of $\D\cap\Z^2$
 defined
as follows: \begin{itemize} \item if $\D_i$ is a
triangle, $\mc{A}_i=\D_i\cap\Z^2$; \item if
$\D_i$ is a parallelogram, $\mc{A}_i$ is the set
of non-special lattice points on $\D_i$.
\end{itemize}

\begin{thm}$ $
\begin{enumerate}
 \item $\mb{V}_{\partial\mc{S}(\D),nodal}$ is a translation of
 a closed subgroup of the torus $\T_{\mc{A}}$;
 \item Its dimension is equal to
 $|\mr{Vertices}(\mc{S}(\D))|-1-|\mr{Parallelograms}(\SD)|$.

\end{enumerate}
\end{thm}
\begin{proof}

The proof of the first statement follows from a
simple adjustment of the proof of Theorem
\ref{thm:SD,rational}. Now the adjacency graph
$\SD^*$ may have  vertices of 4-valency. We
choose  a directed graph on $\SD^*$ which  has no
oriented cycle and no sink at vertices of 3 and
4-valency and also the edges in $\SD^*$ which are
dual to the parallel edges of a parallelogram in
$\SD$ are co-oriented.  Using this directed
graph, we can construct a point in
$\mb{V}_{\partial\mc{S}(\D),nodal}$.
The construction of the closed subgroup $\mb{G}$
is exactly same as
the one given in the proof of Theorem \ref{thm:SD,rational}.\\
The second statement follows from the following
Lemma
 \ref{lem:Vertices-1-Parallelograms}.
\end{proof}
\begin{lem}\label{lem:Vertices-1-Parallelograms} For a nodal subdivision $\mc{S}(\D)$, the following holds true:
 \[2|\mr{Triangles}|+2|\mr{Parallelograms}|-|\mr{IEdges}|=|\mr{Vertices}|-1-|\mr{Parallelograms}|\]\
\end{lem}
\begin{proof}
Let $T:=|\mr{Triangles}|, P:=|\mr{Parallelograms}|,
E:=|\mr{Edges}|, IE:=|\mr{IEdges}|, V:=|\mr{Vertices}|,$ and
$F:=|\mr{\mr{\mr{Faces}}}|$. Then $V-E+F=1, F=T+P,$ and $3T+4P=E+IE$.
Thus $V-1-P=E-T-2P=(3T+4P-IE)-T-2P=2T+2P-IE$.
\end{proof}

\begin{rem}\label{Equation1}
 In terms of  the algebra of tropical cycles  $\mf{A}$ (\S\ref{sec:IT}), we can write:
\[\Tr(\mb{V}_{\partial\SD, nodal})= l(\mb{V})\cdot\Tr(\mb{G}^e),\]
where   $l(\mb{V})$ is the number of components
of $\mb{V}=\mb{V}_{\partial\SD, nodal}$, and
$\mb{G}^e$ is the
 identity component of the closed subgroup
$\mb{G}=\mb{G}_{\partial\SD, nodal}$ of
$\T_{\mc{A}}$.
\end{rem}

\section{Tropical Severi Varieties}

\subsection{Severi varieties} Severi
varieties are very classical varieties which go
back to F. Enriques \cite{Enriques} and F.Severi
\cite{Severi}. As in \S \ref{sec:intermediate} let $X_{\D}$ be the projective
toric
surface constructed from a 2-dimensional convex lattice polygon $\D$, let $\Pro_{\D}$ 
be the tautological linear system on $X_{\D}$ and
let $\T_{\D}$ be the big open torus of
$\Pro_{\D}$.

\begin{defn}\label{defn:Severi} Let
$\delta$ be a non-negative integer with $\delta\le |\mr{Int}(\D)\cap\Z^2|$. \begin{itemize} \item The {\em Severi
variety} $\Sev\subset\Pro_{\D}$ is  the closure of the set of curves with exactly
$\delta$ nodes (ordinary double points)
as their only singularities. \item The {\em very
affine} Severi variety $\Sev^{\circ}$ is the
intersection of the Severi variety $\Sev$ with
the big open torus $\T_{\D}$.
\item The tropicalization of the very affine Severi variety is
called {\em tropical Severi variety} and denoted
by $\Tr(\Sev)$.
\end{itemize}
\end{defn}

It is well known that $\mr{dim}(\Sev)$ is equal to 
$|\D\cap\Z^2|-\delta-1$. We only consider the
case when 
$\Sev^{\circ}$ is dense in 
$\Sev$ and thus they have the same dimension.
This case includes the classical one,  plane
curves of a given degree $d$, equivalently, the
case when the polygon $\D$ is the triangle with
vertices $(0,0), (d,0),(0,d)$, where $d$ is a
positive integer.

\subsection{Patchworking Theory}\label{sec:Patchworking}
In this section we review  Shustin's patchworking theory.
In 1979-80, O. Viro found a patchworking construction for obtaining real nonsingular
projective algebraic hypersurfaces with
prescribed topology. This method was a
breakthrough in Hilbert's 16th problem. In the
early 1990's, E. Shustin suggested to use the
patchworking construction for tracing other
properties of objects defined by polynomials, for
example, prescribed singularities of algebraic
hypersurfaces and many others\cite{Shustin1,
Shustin2}.  He starts with a modified version of
the patchworking construction, which allows one
to keep singularities in the patchworking
deformation. An important difference with respect
to the original Viro method is that singularities
are not stable in general, and thus one has to
modify the Viro deformation and impose certain
transversality conditions.

 The following is a version of Shustin's patchworking theory
about curves on toric surfaces, summarized for
the purposes of this paper \cite[\S
3.7]{Shustin2}. (Note the conventions in \S \ref{sec:Tropicalization}.)

\begin{itemize}\item Let $\SD$ be a regular subdivision of $\D$ with
$\mr{rank}(\SD)=\mr{dim}(\Sev)$.\\ Suppose that  $\SD$ is
simple and nodal, then there exists a c-vector
$\w$ of  $\Sev$ such that $\D_{\w}=\SD$.
\item
Let $\w:\D\cap\Z^2\rightarrow\Z$ be an
integral-valued function on $\D\cap\Z^2$ such
that\\ $\mr{rank}(\w)=\mr{dim}(\Sev)$ and $\D_{\w}$ is
simple-\textit{triangular}.

(Enumeration 1) If we fix the coefficients
$c_b\in\mb{C}\setminus\{0\}$ for $b\in
\mr{Vertices}(\D_{\w})$,  then the number of
$F\in\mb{V}=\mb{V}_{\partial{\D_{\w}},nodal}$
with $F(b)= c_b$ is equal to \[\frac{\prod
 2\mr{area}(\mr{Triangles})}{\prod\mr{length(\mr{Edges}(\D_{\w}))}},\label{first} \] where the numerator
 stands for the product of  twice the (Euclidean) area of each triangle in
 $\D_{\w}$ and the denominator is
  the product of the
lattice lengths of the  edges.

(Enumeration 2) If we fix the coefficients
$c_b(t)=\bar{c_b}t^{\w(b)}+l.o.t.\in\mb{K}\setminus\{0\}$
for $b\in \mr{Vertices}(\D_{\w})$, then the number of
$f\in\Sev(\K)$ with $f(b)= c_b(t)$ and $\tau_f$
dual to $\D_{\w}$ is equal to
\[\prod
 2\mr{area}(\mr{Triangles}).\label{second}\] (In Shustin's notations in \cite[\S 3.7]{Shustin2}, given $c_b(t)$
 for $b\in \mr{Vertices}(\D_{\w})$, the number of possible $A$ (amoeba) is 1,  the number of
  possible $F$ (initial terms of coefficients of $f$) is $\frac{\prod
 2\mr{area}(\mr{Triangles})}{\prod\mr{length(\mr{Edges}(\D_{\w}))}}$, and the number of
 possible $R$ (deformation patterns) is $\prod\mr{length(\mr{Edges}(\D_{\w}))}$. Thus the number of possible $(A,F,R)$ is
  equal to $\prod
 2\mr{area}(\mr{Triangles})$ and each of them  gives rise to a unique $f\in\Sev(\K)$.)

\end{itemize}

\begin{rem}\label{rem:simple-nodal}
 In fact the enumerations above hold when the
subdivision $\D_{\w}$ is simple-\textit{nodal},
which was the case Shustin worked on. In this
case, we replace $\mr{Vertices}(\D_{\w})$ by a subset
$\mc{B}$ with $|\mc{B}|=\mr{dim}(\Sev)+1$ so that  for
any $F\in\mb{V}$ fixing coefficients for
$b\in\mc{B}$ uniquely determines the other
coefficients for $b\in
 \mr{Vertices}(\D_{\w})\setminus\mc{B}$. Also the denominator of the formula (\ref{first}) should be adjusted as follows: $\widetilde{\prod}\mr{length(\mr{Edges}(\D_{\w}))}$, the product of the
lattice lengths of the  edges  which are
representatives of each equivalence class in
$\mr{Edges}(\D_{\w})$, where we define an equivalence
relation as follows: let $e\sim e'$ if $e$ and
$e'$ are the parallel edges of a parallelogram in
$\D_{\w}$ and extend it by transitivity.

\end{rem}

\subsection{Initial schemes of very affine Severi varieties}\label{sec:Initial Schemes of Sev}
Let $\w\in\Z^{\D\cap\Z^2}$ be an integral vector.
As in \S\ref{sec:constant}, we have the initial
scheme $\n\w\Sev\subset\T_{\D}$ of the very
affine Severi variety $\Sev^{\circ}$. Also, $\w$
can be identified with   an integral-valued
function on the set of lattice points
$\D\cap\Z^2$ on $\D$ and we get the regular
subdivision $\D_{\w}$ of $\D$ constructed from
$\w$. We are going to find a description of
$\n\w\Sev$ with respect to $\D_{\w}$.
Let us begin with an example which is simple
since the Severi variety is a hypersurface:
\begin{eg}
Let $\D$ be as described on the right.\\
\begin{parbox}
{7cm}{A general polynomial with Newton polygon
$\D$ is written as $f=ay^2+bx^2y+cxy+dy+e.$}
\end{parbox}
\setlength{\unitlength}{.25mm}
\begin{picture}(0,0)(-160,-5)
\put(-110,22){$\Delta:$}
\curve(-75,0,-35,20)\curve(-75,0,-75,40)\curve(-75,40,-35,20)
\put(-77.5,-2.5){\tiny{$\bullet$}}\put(-84,-8.2){\tiny{$e$}}
\put(-77.5,17.7){\tiny{$\bullet$}}\put(-84,11.2){\tiny{$d$}}
\put(-77.5,37.5){\tiny{$\bullet$}}\put(-84,31.2){\tiny{$a$}}
\put(-57.5,17.5){\tiny{$\bullet$}}\put(-57.5,11.2){\tiny{$c$}}
\put(-37.5,17.5){\tiny{$\bullet$}}\put(-37.5,9.2){\tiny{$b$}}
\put(-35,-7){\tiny{$\Real^2$}}\end{picture}\\
 We
consider the curves $V(f)$ with one singular
point, that is, $Sev(\D,1)$.\\ The hypersurface
$Sev(\D,1)$ is defined by a polynomial
$D_{\mc{A}}$ which is called
$\mc{A}$-discriminant, where $\mc{A}=\D\cap\Z^2$:
\[D_{\mc{A}}= 16b^2d^2-8bc^2d+c^4-64ab^2e\]
The figure  below on the left represents the
Newton polytope of $D_{\mc{A}}$ and
$\Tr(Sev(\D,1))$ which is modded out by the
3-dimensional linearity space. The figure on the
right shows the connection to the subdivisions of
$\D$.\\
\setlength{\unitlength}{.2mm}
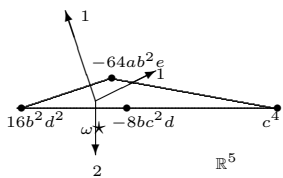
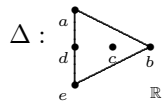
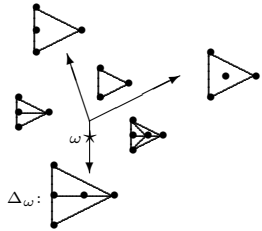
\begin{picture}
(0,0)(-60,0)
 \curve(0,-70,170,-70)\curve(0,-70,60,-50)\curve(170,-70,60,-50)
\put(67.5,-72.5){\tiny{$\bullet$}}\put(60.5,-82){\tiny{$-8bc^2d$}}
\put(-2.5,-72.5){\tiny{$\bullet$}}\put(-9.5,-82){\tiny{$16b^2d^2$}}
\put(167.5,-72.5){\tiny{$\bullet$}}\put(160.5,-82){\tiny{$c^4$}}
\put(57.5,-52.5){\tiny{$\bullet$}}\put(47,-42.5){\tiny{$-64ab^2e$}}
\put(130,-110){\tiny{$\Real^5$}}

\put(50,-65){\vector(0,-1){36}}\put(48,-115){\tiny{2}}\put(48,-87){$\star$}\put(40,-87){\tiny{$\w$}}
\put(50,-65){\vector(-1,3){20}}\put(40,-12){\tiny{1}}
\put(50,-65){\vector(2,1){40}}
\put(90,-50){\tiny{1}}
\end{picture}

\begin{picture}
(0,0)(-350,-20)
\put(50,-65){\vector(0,-1){35}}\put(46,-80){$\star$}
\put(32,-80){\tiny{ $\w$}}
\put(50,-65){\vector(-1,3){15}}
\put(50,-65){\vector(2,1){60}}\put(-5,-120){\tiny{$\D_{\w}$:}}
\end{picture}
\begin{picture}
(0,0)(-345,0)
\curve(55,-30.5,55,-10.5)\curve(55,-30.5,75,-20.5)
\curve(75,-20.5,55,-10.5)
\put(53,-32){\tiny{$\bullet$}}
\put(53,-12.5){\tiny{$\bullet$}}
\put(72,-22.5){\tiny{$\bullet$}}
\end{picture}
\begin{picture}(0,0)(-360,-20)
\curve(-20,-70,0,-60)\curve(-20,-70,-20,-50)\curve(0,-60,-20,-50)
\put(-22,-72){\tiny{$\bullet$}}
\put(-2,-62){\tiny{$\bullet$}}
\put(-22,-52){\tiny{$\bullet$}}
\put(-22,-62){\tiny{$\bullet$}}
\curve(-20,-60,0,-60)
\end{picture}
\begin{picture}(0,0)(-250,5)
\curve(160,-60,180,-50)\curve(160,-60,160,-40)\curve(180,-50,160,-40)
\put(158,-62){\tiny{$\bullet$}}
\put(178,-52){\tiny{$\bullet$}}
\put(158,-52){\tiny{$\bullet$}}
\put(158,-42){\tiny{$\bullet$}}
\put(168,-52){\tiny{$\bullet$}}
\curve(160,-50,180,-50) \curve(160,-60,170,-50)
\curve(160,-40,170,-50)
\end{picture}
\begin{picture}(0,0)(-330,-60)
\curve(10,-60,10,-30)\curve(10,-60,40,-45)\curve(40,-45,10,-30)
\put(8,-62){\tiny{$\bullet$}}
\put(8,-32){\tiny{$\bullet$}}
\put(38,-47){\tiny{$\bullet$}}
\put(8,-47){\tiny{$\bullet$}}
\end{picture}
\begin{picture}(0,0)(-350,-20)
\curve(100,-50,100,-20)\curve(100,-50,130,-35)\curve(130,-35,100,-20)
\put(98,-52){\tiny{$\bullet$}}
\put(98,-22){\tiny{$\bullet$}}
\put(128,-38){\tiny{$\bullet$}}
\put(108,-38){\tiny{$\bullet$}}
\end{picture}
\begin{picture}(0,0)(-250,0)
\curve(90,-75,90,-115)\curve(90,-115,130,-95)\curve(130,-95,90,-75)
\put(88,-77){\tiny{\tiny{$\bullet$}}}
\put(88,-117){\tiny{$\bullet$}}
\put(128,-97){\tiny{$\bullet$}}
\put(88,-97){\tiny{$\bullet$}}
\put(108,-97){\tiny{$\bullet$}}
\curve(90,-95,130,-95)
\end{picture}\\\\\\\\\\\\\\
Choose an integral vector $\w$ on one of the rays
of $\Tr(Sev(\D,1))$ as shown on the figure above.
The corresponding initial scheme $\n\w Sev(\D,1)$
is defined by the polynomial
$16b^2d^2-8bc^2d+c^4=(4bd-c^2)^2$. That is, $\n\w
Sev(\D,1)$ is a non-reduced scheme, a translation
of the torus $V(bd-c^2)$ with multiplicity $2$.
Let us look at the corresponding subdivision
$\D_{\w}$. We can find that $l(\mb{V})$,
the number of components of
$\mb{V}=\mb{V}_{\partial\D_{\w}, nodal}$, is
equal to $1$. Also we see that $\D_{\w}$ has one
interior edge of lattice length $2$. Therefore,
the number of translations of a subtorus (counted
with multiplicity)  in the initial scheme $\n\w
Sev(\D,1)$ coincides with   the product of
$l(\mb{V})$ and the lattice length of the
interior edge of $\D_{\w}$. The main theorems of
this paper presented below show that this
description of the initial schemes of Severi
varieties holds true in general.
\end{eg}
 The following theorem gives a description of the support of $\Tr(\Sev)$.
\begin{thm}\label{thm:support}
If the rank of $\w$ is strictly larger than
$\mr{dim}(\Sev)$, then  $\w$ is not a c-vector of
$\Sev$, that is, $\n\w\Sev=\emptyset$.
\end{thm}
\begin{proof} Suppose $\w$ is a c-vector of $\Sev$ with
$\mr{rank}(\w)\ge \mr{dim}(\Sev)$. In \S\ref{sec:nodal
curves}, we showed that $\w$ must satisfy
Shustin's combinatorial characterization, i.e.,
the subdivision $\D_{\w}$ is simple-nodal and
$\mr{rank}(\w)=\mr{dim}(\Sev)$. Thus the rank of any
c-vector of $\Sev$ cannot by strictly larger than
$\mr{dim}(\Sev)$.
\end{proof}
Also if $\w$ is a c-vector of $\Sev$ with
maximal rank $\mr{dim}(\Sev)$, Shustin's combinatorial characterization says that  the regular
subdivision $\D_{\w}$ must be simple and nodal.
Furthermore, if we impose one more restriction on
the parallelograms in $\D_{\w}$, we can obtain a
{\em complete} description of  initial schemes
$\n\w\Sev$ as presented in the next theorem following lemma which provides a geometric characterization of initial schemes.

\begin{lem}\label{lem:geometric}Let $X$ be a subvariety of an algebraic torus $\T$. 
The set of closed points of  $\n\w X$  is equal
to
\[\{\bar{z}\in\T: \text{ there exists }z=\bar{z}t^{\w}+l.o.t.\in
X(\K)\subset\T(\K)\},\] where $l.o.t.$ stands for
``lower order terms'' and $z$ is in
vector-notation.
\end{lem}

\begin{proof} The proof for the inclusion $\subset$ can be found in \cite[Lemma 4.15]{Katz2}, \cite[Proposition
4]{Kazarnovskii1}, \cite{Payne}.
 Let us consider the other inclusion $\supset$. Suppose $z=\bar{z}t^{\w}+l.o.t.\in
 X(\K)$ and let $f\in I(X)$. It is enough to show that $\n\w f(\bar{z})=0$, which follows
 from the fact that $\n\w f(\bar{z})$ is the constant term of $f(z)\cdot t^{-\gamma}\in\C[t]$, where
 $\gamma$ is the $t-\w$-degree of $f$.
\end{proof}

\begin{thm}\label{thm:No nonprimitive}
Let $\w$ be a c-vector of  $\Sev$. Suppose that
$\w$ satisfies the following conditions:
\begin{itemize}
\item The rank of $\w$ is maximal, that is, $\mr{rank}(\w)=\mr{dim}(\Sev)$;
\item The regular subdivision $\D_{\w}$ has no non-primitive parallelogram.
\end{itemize}
Then the following hold true:
\begin{enumerate}
\item As the sets of closed points, $\n\w\Sev$ is equal to
$\mb{V}=\mb{V}_{\partial\D_{\w}, nodal}$. Thus,
$\n\w\Sev$ is a union of finitely many
translations of the torus $\mb{G}^e$, the
identity component of
$\mb{G}=\mb{G}_{\partial\D_{\w}, nodal}$.
\item The weight of $\w$ on $\n\w\Sev$, that is, the number of such translations of the
torus $\mb{G}^e$ counted with multiplicity, is
equal to
\[\bsm_{\Sev}(\w)=l(\mb{V})\cdot
\widetilde{\prod}\mr{length(\mr{Edges}(\D_{\w}))},\]
where
\begin{enumerate}
\item $l(\mb{V})$ is the number of connected components of $\mb{V}$;
\item $\widetilde{\prod}\mr{length(\mr{Edges}(\D_{\w}))}$ is the product of the
lattice lengths of the  edges  which are
representatives of each equivalence class in
$\mr{Edges}(\D_{\w})$, where we define an equivalence
relation as follows: let $e\sim e'$ if $e$ and
$e'$ are the parallel edges of a parallelogram in
$\D_{\w}$ and extend it by transitivity.
\end{enumerate}

\end{enumerate}
\end{thm}
\begin{proof} We prove the first statement. Applying  Lemma \ref{lem:geometric} to
 our case $X=\Sev$, we  see that
$\bar{c}=(\bar{c}_a)_{a\in\D\cap\Z^2}\in\n\w\Sev$
if and only if there exists a 1-parameter
equisingular family of nodal curves with $\delta$
nodes defined by
\[f_{(t)}(x,y)=\sum_{a=(a_1,a_2)\in\D\cap\Z^2}c_a(t)x^{a_1}y^{a_2}\] such
that \[c_a(t)=\bar{c}_at^{\w_a}+l.o.t.\]
Thus, $\bar{c}$ is closely related to the
tropicalization of $f_{(t)}$. Let us recall the
definition of the tropicalization of
$f_{(t)}$:
Let $\nu_{f_{(t)}}$ be the concave hull of
$\w=Val_{f_{(t)}}$ and rewrite $f_{(t)}$ with
respect to $\nu_{f_{(t)}}$:
\[f_{(t)}(x,y)=\sum_{a=(a_1,a_2)\in\D\cap\Z^2}c_a(t)x^{a_1}y^{a_2}\] such
that
\[c_a(t)=c_a^{\circ}t^{\nu_{f_{(t)}}(a)}+l.o.t.,\]
where $c_a^{\circ}$ is a complex number such that
$c_a^{\circ}=\bar{c}_a$ exactly when
$\nu_{f_{(t)}}(a)=\w(a)$ and otherwise
$c_a^{\circ}=0$.
This collection of complex numbers,
$c^{\circ}=\{c_a^{\circ}:a\in\D\cap\Z^2\}$,
together with the regular subdivision
$\D_{\nu_{f_{(t)}}}=\D_{\w}:\D_1\cup\cdots\D_m$
of the Newton polygon $\D$ of $f_{(t)}$ gives
rise to a collection of complex polynomials
$f_1,\dots,f_m$ with $\mr{Newton}(f_i)=\D_i$ for
$i=1,\dots,m$.
Now Shustin's geometric characterization
\ref{thm:Shustin's characterization} implies that
under the hypothesis of our theorem, if
$\bar{c}\in\n\w\Sev$, then
$\bar{c}=c^{\circ}\in\mb{V}_{\partial\D_{\w},
nodal}$. (Notice that in our case that there is
no non-primitive parallelogram in the subdivision
$\D_{\w}$, $c^{\circ}(a)\ne 0$ for all
$a\in\D\cap\Z^2$, which also
implies that $\nu_{f_{(t)}}=\w$, that is, $\w$ should be concave.)
Thus, $\n\w\Sev$ is a subset of
$\mb{V}_{\partial\D_{\w}, nodal}$. The other
inclusion follows easily from Shustin's
patchworking theory \S\ref{sec:Patchworking}.

Now we prove the second statement.
 Let $\mc{B}$ be a subset of
$\mr{Vertices}(\D_{\w})$ with the properties given in
Remark \ref{rem:simple-nodal}, and let
$\mb{L}_{\mc{B}}$ be the $n-(\mr{dim}(\Sev)+1)$
 dimensional coordinate subspace of $\T_{\D}$ defined by
the equations $x_b=1, \quad b\in\mc{B}$. Then
Shustin's first enumeration (\ref{first}) deduces
the following:
\[(\Tr(\mb{V})\cdot\Tr(\mb{L}_{\mc{B}}))=\frac{\prod 2\mr{area}(\mr{Triangles})}{\widetilde{\prod}\mr{length(\mr{Edges}(\D_{\w}))}}.\]
Moreover, from the second enumeration
(\ref{second}), we obtain the following:
\[((\Tr(\n\w\Sev)\cdot\Tr(\mb{L}_{\mc{B}})))=\prod
 2\mr{area}(\mr{Triangles}).\]
Thus,  \[\begin{array}{ccl}
\bsm_{\Sev}(\w)(\Tr(\mb{G}^e)\cdot\Tr(\mb{L}_{\mc{B}}))
&=&(\Tr(\n\w\Sev)\cdot\Tr(\mb{L}_{\mc{B}}))\\
&=&\prod
 2\mr{area}(\mr{Triangles})\\ &=&\widetilde{\prod}\mr{length(\mr{Edges}(\D_{\w}))} \cdot\frac{\prod
 2\mr{area}(\mr{Triangles})}{\widetilde{\prod}\mr{length(\mr{Edges}(\D_{\w}))}}\\
&=&\widetilde{\prod}\mr{length(\mr{Edges}(\D_{\w}))}(\Tr(\mb{V})\cdot\Tr(\mb{L}_{\mc{B}}))\\
&=&l(\mb{V})\widetilde{\prod}\mr{length(\mr{Edges}(\D_{\w}))}(\Tr(\mb{G}^e)\cdot\Tr(\mb{L}_{\mc{B}}))
\end{array}\]
Thus we obtain
\[\bsm_{\Sev}(\w)=l(\mb{V})\widetilde{\prod}\mr{length(\mr{Edges}(\D_{\w}))}.\]

\end{proof}

Now we consider the case that there is a
non-primitive parallelogram  in the subdivision
$\D_{\w}$. We need to consider
a certain projection. Let us begin with a general
setting.  Let $\SD$ be a nodal subdivision which
may have non-primitive parallelograms. As we
studied in \S \ref{sec:SD,nodal}, in this case,
the variety $\mb{V}_{\partial\SD,nodal}$ is
contained in every coordinate hyperplane $H_a$ of
the ambient projective space $\Pro_{\D}$ defined
by $x_a=0$, where $a$ is a special  point in a
non-primitive parallelogram. In
 particular, $\mb{V}_{\partial\SD,nodal}$ is disjoint from the big open torus
 $\T_{\D}$. Let $H_{\SD}$ be the intersection of all such coordinate
 hyperplanes $H_a$.  Let $\T_{\SD}\subset H_{\SD}$
 be the big open torus in $H_{\SD}$ so that
 $\mb{V}_{\partial\SD,nodal}\subset\T_{\SD}$. Let $\pi_{\SD}$ be
 the projection from $\Pro_{\D}$ to $H_{\SD}$, \[\pi_{\SD}:\Pro_{\D}\rightarrow H_{\SD}.\]
 Now we consider the case when the subdivision $\SD$ is given by a
 c-vector $\w$ of $\Sev$ with maximal rank, that is, $\SD=\D_{\w}$
 and $\mr{rank}(\w)=\mr{dim}(\Sev)$. We impose one more condition that  $\w$ is a {\em regular} point of the
 tropical Severi variety  $\Tr(\Sev)$, that is,
 $\Tr(\Sev)$ coincides with  an affine space of dimension
 $\mr{dim}(\Sev)$ locally near $\w$. (Warning: the maximality of rank of $\w$ does
 not necessarily
 imply that $\w$ is a regular point of $\Tr(\Sev)$.)
Then we obtain the following theorem.

\begin{thm}\label{thm:Existence of nonprimitive}
Let $\w$ be a c-vector of $\Sev$. Suppose that
$\w$ satisfies the following conditions:
\begin{itemize}
\item The rank of $\w$ is maximal, that is, $\mr{rank}(\w)=\mr{dim}(\Sev)$ ;
\item $\w$ is a regular point in $\Tr(\Sev)$.
\end{itemize}
Then the following statements hold true:
\begin{enumerate}
\item The projection $\pi_{\D_{\w}}$ is a  bijection   from  $\n\w\Sev$
to $\mb{V}=\mb{V}_{\partial\D_{\w},nodal}$. \item
The initial scheme $\n\w\Sev$ is the union of
finitely many translations of a torus
$\mb{G}^{e*}$ of dimension  $\mr{dim}(\Sev)$ which is
sent to $\mb{G}^e$ by the projection
$\pi_{\D_{\w}}$, where $\mb{G}^e$ is the identity
component of $\mb{G}=\mb{G}_{\partial\D_{\w},
nodal}$.
\item The weight of $\w$ on $\n\w\Sev$, that is, the number of such translations   of the
torus $\mb{G}^{e*}$ counted with multiplicity, is
equal to
\[\bsm_{\Sev}(\w)=l(\mb{V})\cdot
\widetilde{\prod}\mr{length(\mr{Edges}(\D_{\w}))},\]
as defined in the previous theorem.

\end{enumerate}

\end{thm}

\begin{proof}
The first and second statements are deduced
straight-forwardly from the conditions on $\w$.
 Let us show the last statement. It is a slight adjustment of the
proof in the previous theorem adding the
consideration of  the
projection $\pi_{\D_{w}}$.\\
 Shustin's first
enumeration (\ref{first}) deduces the following:
\[(\Tr(\mb{V})\pi_{*}(\Tr(\mb{L}_{\mc{B}})))=\frac{\prod 2\mr{area}(\mr{Triangles})}{\widetilde{\prod}\mr{length(\mr{Edges}(\D_{\w}))}},\]
where $\pi_{*}:\Real^n\rightarrow\Real^{|\mc{A}|}$ is the projection corresponding to $\pi=\pi_{\D_{\w}}$. 
Moreover, from the second enumeration
(\ref{second}), we obtain the following:
\[(\Tr(\n\w\Sev)\cdot\Tr(\mb{L}_{\mc{B}}))=\prod
 2\mr{area}(\mr{Triangles}).\]
Thus,  \[\begin{array}{ccl}
\bsm_{\Sev}(\w)(\Tr(\mb{G}^{e*})\cdot\Tr(\mb{L}_{\mc{B}}))
&=&(\Tr(\n\w\Sev)\cdot\Tr(\mb{L}_{\mc{B}}))\\
&=&\prod
 2\mr{area}(\mr{Triangles})\\ &=&\widetilde{\prod}\mr{length(\mr{Edges}(\D_{\w}))} \cdot\frac{\prod
 2\mr{area}(\mr{Triangles})}{\widetilde{\prod}\mr{length(\mr{Edges}(\D_{\w}))}}\\
&=&\widetilde{\prod}\mr{length(\mr{Edges}(\D_{\w}))}(\Tr(\mb{V})\cdot\pi_{*}(\Tr(\mb{L}_{\mc{B}})))\\
&=&l(\mb{V})\widetilde{\prod}\mr{length(\mr{Edges}(\D_{\w}))}(\Tr(\mb{G}^{e})\cdot\pi_{*}(\Tr(\mb{L}_{\mc{B}})))\\
&=&l(\mb{V})\widetilde{\prod}\mr{length(\mr{Edges}(\D_{\w}))}(\Tr(\mb{G}^{e*})\cdot\Tr(\mb{L}_{\mc{B}}))
\end{array}\]
The last equality can be seen easily by
considering the projection $\pi_{*}$ as follows:
by choosing a coordinate system,
$(\Tr(\mb{G}^e)\cdot\pi_{*}(\Tr(\mb{L}_{\mc{B}})))$
is the determinant of a matrix
$\left(\begin{tabular}{c|c} $M_1$ &
$M_2$\end{tabular}\right)$, where $M_1$ and $M_2$
are found from lattice bases of $\Tr(\mb{G}^e)$
and $\pi_{*}(\Tr(\mb{L}_{\mc{B}}))$,
respectively. Then
$(\Tr(\mb{G}_{\circ}^*)\cdot\Tr(\mb{L}_{\mc{B}}))$
is the determinant of a matrix of the form of
$\left(\begin{tabular}{c|c|c}$M_1$&$M_2$&$0$\\\hline$*$&$0$&$I$\end{tabular}\right)$,
where $I$ is the identity matrix. Therefore\\
$(\Tr(\mb{G}^e)\cdot\pi_{*}(\Tr(\mb{L}_{\mc{B}})))=
(\Tr(\mb{G}^{e*})\cdot\Tr(\mb{L}_{\mc{B}}))$.\\
Thus,
\[\bsm_{\Sev}(\w)=l(\mb{V})\widetilde{\prod}\mr{length(\mr{Edges}(\D_{\w}))}.\]
\end{proof}
\begin{rem}
 As a corollary, if $\n\w\Sev\ne\emptyset$ and the subdivision $\D_{\w}$  is either simple or nodal but not both,
 then $\mr{rank}(\w) < r$.
\end{rem}
\subsection{The degrees of Severi varieties}
\label{sec:Severi degree}
In this section, we study  Mikhalkin's
Correspondence theorem with respect to tropical
intersection theory. Let us review this theorem.

\begin{defn}\cite[Definition 2.41]{IMS}\label{defn:GeneralPosition}$ $\begin{enumerate}
\item Let $\SD$ be a  subdivision of $\D$.
We say that the distinct points
$x_1,\dots,x_{\zeta}\in\mb{Q}^2$ are in
{\em $\SD$-general position}, if the condition for
tropical curves to pass through
$x_1,\dots,x_{\zeta}$ (``base-point-condition'')
cuts out the tropical cone $\mc{T}C(\SD)$ either
the empty set, or a polyhedron of codimension
$\zeta$.
\item We say that the distinct points $x_1,\dots,x_{\zeta}$ are in $\D$-{\em general
position (or simply, generic points)}, if they are $\SD$-general for all
subdivisions $\SD$ of $\D$.
\end{enumerate}
\end{defn}
\begin{lem}\cite[Lemma 2.42]{IMS} For any given convex lattice polygon $\D$, the set of
$\D$-general configurations $x_1,\dots,x_{\zeta}$
is dense in $(\mb{Q}^2)^{\zeta}$.
\end{lem}
To present the correspondence theorem we need one
more numeric invariant assigned to a subdivision
$\SD$ of a polygon $\D$, besides the rank of
$\SD$: suppose $\SD$ is nodal, that is, the
subpolygons are either triangles or
parallelograms.  Then the (Mikhalkin's) {\em
multiplicity} of $\SD$ is by definition
\[\mu(\SD):= \prod 2\mr{area}(\mr{Triangles}),\]  the
product of twice areas of all the triangles in
$\SD$.

\begin{thm}[Mikhalkin's Correspondence Theorem] \cite[Theorem 2.43]{IMS},\cite{Mikhalkin}\label{thm:correspondence}

Let $\mc{P}$ be a set of
$r=\mr{dim}(\Sev)$ points in $\Real^2$ which are in $\D$-general position.
Then   
\[ \mr{degree}(\Sev)=\sum_{\w} \mu(\w),\] where the sum runs over all
 tropical  curves $\tau_{\w}$ of
degree $\D$ passing through all the points in
$\mc{P}$. ($\mu(\w)$ is by
definition $\mu(\D_{\w})$.)

\end{thm}

First, we show that the set $\mc{S}$  of such
tropical curves described above is in one-to-one
correspondence with the set-theoretic transversal
intersection of two tropical varieties (Definition \ref{defn:linear space}),
$\Tr(\Sev)\cap\Tr(\mc{L}(\bs{p}))$.
Then, we show that Mikhalkin's multiplicity of
any curve in the set $\mc{S}$ is equal to the
tropical intersection multiplicity of the
corresponding point in the intersection
$\Tr(\Sev)\cap\Tr(\mc{L}(\bs{p}))$. Thus,
Mikhalkin's enumeration is equal to the
computation of the degree
$(\Tr(\Sev)\cdot\Tr(\mc{L}(\bs{p})))$.
 Let us begin with the definition of
$\mc{L}(\bs{p})$.
\begin{defn}\label{defn:linear space}
 Let $\bs{p}=\{p_1,\dots, p_{\zeta}\}\subset \TK^2$ be a finite set of points in $\TK^2$.
 Define $\mc{L}(\bs{p})\subset\Pro_{\D}(\K)$ to  be the parameter space of algebraic curves on
 the toric surface $X_{\D}(\K)$ passing through all the points in $\bs{p}$. This parameter space $\mc{L}(\bs{p})$ is
 the complete intersection of hyperplanes $\mc{H}_{p_j}\subset\Pro_{\D}(\K)$ defined by the condition of
  passing through the point $p_j,(j=1,\dots, \zeta).$ The intersection of $\mc{L}(\bs{p})$ with the big open torus $\T_{\D}$ is again denoted by $\mc{L}(\bs{p})$.
\end{defn}
\begin{lem}\label{lem:point-condition}
 An integral vector $\w$ is a c-vector of the hyperplane $\mc{H}_q$, i.e., $\w\in
 \Tr(\mc{H}_q)$,
 if and only if the tropical curve $\tau_{\w}$ passes through the point  $Val(q)=(Val(q_1), Val(q_2))$, where $q=(q_1,q_2)\in\TK^2$.
\end{lem}
\begin{proof} Let $q_1=\alpha t^{m}+l.o.t.$ and $q_2=\beta t^n + l.o.t.$ and thus $Val(q)=(m,n)$. Now $\mc{H}_q$ is the
hyperplane defined by the linear polynomial
\[\sum_{a=(a_1,a_2)\in \D\cap\Z^2} (\alpha t^{m}
+ l.o.t.)^{a_1}(\beta t^{n} + l.o.t)^{a_2}
c_a=0\] in the variables $c_a, \quad
a\in\D\cap\Z^2$. Thus, the support of
$\Tr(\mc{H}_q)$ is the corner locus of the map:
 \[(x_a)_{a\in\D\cap\Z^2}\mapsto max_{a\in\D\cap\Z^2} \{x_a + (m,n)\cdot(a_1,a_2)\}\]
Also,  the tropical curve $\tau_{\w}$ is by
definition the corner locus of the map:
\[(X,Y)\mapsto max_{a\in\D\cap\Z^2}\{ \w_a + (X,Y)\cdot(a_1, a_2)\}\]
The statement follows in a straightforward way.
\end{proof}
From the Lemma \ref{lem:point-condition} above,
we see that $\w$ is a c-vector of
$\mc{L}(\bs{p})$ if and only if the tropical
curve $\tau_{\w}$ passes through all the points
$Val(p_1),\dots, Val(p_{\zeta})$ in $\Real^2$.
%

%
%
\begin{thm} \label{thm:Severi degree}Let
  $\bs{p}=\{p_1,\dots,p_r\}\in(\TK^2)^r$ be a configuration of  $r$ generic
   points in
  $\TK^2$ so that
$\bs{Val(p)}=\{Val(p_1),\dots,Val(p_r)\}\in(\Rational^2)^r$
 is in $\D$-general position, 
  $\mathrm{Trop}(\mc{L}(\bs{p}))\cap\Tr(\Sev)$
 is a transversal intersection and the linear system $\ref{independent}$ is independent,  where $r=\mr{dim}(\Sev)$.
Then the following statements hold true:
\begin{enumerate}
\item The intersection $\Tr(\mc{L}(\bs{p}))\cap\Tr(\Sev)$ is in
one-to-one correspondence with the set of
 tropical curves $\tau_{\w}$  passing
through all the points in $\bs{Val(p)}$.
\item The  extrinsic intersection multiplicity (\S 2.4. \ref{extrinsic})  of
$\mathrm{Trop}(\mc{L}(\bs{p}))$ and $\Tr(\Sev)$\\
at
$\w\in\mathrm{Trop}(\mc{L}(\bs{p}))\cap\Tr(\Sev)$
is
\[\xi(\w;\mathrm{Trop}(\mc{L}(\bs{p})),\Tr(\Sev))=\frac{\prod
2\mr{area}(\mr{Triangles})}{l(\mb{V})\cdot
\widetilde{\prod}\mr{length(\mr{Edges})}}.\]
\item The tropical intersection multiplicity (\S 2.4. \ref{tropical}) of $\mathrm{Trop}(\mc{L}(\bs{p}))$ and $\Tr(\Sev)$\\ at
$\w\in\mathrm{Trop}(\mc{L}(\bs{p}))\cap\Tr(\Sev)$
is equal to Mikhalkin's multiplicity of the
tropical curve $\tau_{\w}$:
\[\bsm(\w;\mathrm{Trop}(\mc{L}(\bs{p})),\Tr(\Sev))=
\prod 2\mr{area}(\mr{Triangles}).\]
\end{enumerate}
\end{thm}
\begin{proof} Let us prove the first statement:
 Since the set of  c-vectors are open dense in the tropicalization of a variety, we
can assume that the  intersection points in
$\Tr(\mc{L}(\bs{p}))\cap\Tr(\Sev)$ are all
c-vectors, in particular, they are all {\em
rational} vectors. From  Lemma
\ref{lem:point-condition} and Shustin's
combinatorial characterization Theorem
\ref{thm:Shustin's characterization}, the first
statement follows.

To prove the second statement,
 we find  neighborhoods of $\w$ in $\Tr(\Sev)$ and in
$\mathrm{Trop}(\mc{L}(\bs{p}))$.
Then we compute the volume of the corresponding
principal parallelepiped, which is by definition
the extrinsic intersection multiplicity of
$\mathrm{Trop}(\mc{L}(\bs{p}))$ and $\Tr(\Sev)$
at $\w$ (\S 2.4. \ref{extrinsic}).
 Since the points are in $\D$-general
 position and the base-point-condition cuts out the cone $\mc{K}(\TSD_{\w})$ non-empty  ($ cc(\w)_{\Z}\in\mc{K}(\TSD_{\w})$),
 the rank of $\w$  must be at least
 $r=\mr{dim}(\Sev)$ and so is equal to $r$, since $\w$ is a c-vector of
 $\Sev$. Also $\w$ is a regular point of both $\Tr(\Sev)$ and
 $\Tr(\mc{L}(\bs{p}))$, being a traversal-intersection-point.
Thus, by Theorem \ref{thm:Existence of
nonprimitive}, near $\w$, the support of
$\Tr(\Sev)$ is equal to
$\Tr(\mb{G}^{e*})$, which is a $r$-dimensional  linear space. 
Now let us consider a full-dimensional 
neighborhood of $\w$ in
$\mr{Trop}(\mc{L}(\bs{p}))$. The tropical curves
corresponding to points in such neighborhood are
tropical curves passing through the $r$ points
$Val(p_1),\dots,Val(p_r)$ and their degrees are
subsets of $\D$. In particular, the tropical
curve $\tau_{\w}$ also passes through the points.
Since $\mr{rank}(\w)=r$,  they lie  on $r$ distinct edges of
$\tau_{\w}$ which correspond to some $r$ edges of
the subdivision $\D_{\w}$. If $\sigma_i\in
\mr{Edges}(\D_{\w})$ correspond to a point $Val(p_i)$
and $a_i, a_i'$ are the endpoints of
$\sigma_i,\quad 1\le i\le r$, then we have the
following linear conditions on $\w(a_i)$ and
$\w(a_i')$:
\[\label{independent}\w(a_i)-\w(a_i')=(a_i'-a_i)\cdot Val(p_i), \quad(i=1,\dots,r.)\]
 Let $\mc{B}$ be the set of vertices
of $\D_{\w}$ which are involved in this independent linear
system. Thus the values of $\w$ on $\mc{B}$ are
fixed and ones on $(\D\cap\Z^2)\setminus\mc{B}$
can be any values. Thus we can see that near
$\w$, the support of $\Tr(\mc{L}(\bs{p}))$ is
equal to $\Tr(\mb{L}_{\mc{B}})$, where
$\mb{L}_{\mc{B}}$ is defined in the proof of
Theorem \ref{thm:No nonprimitive}.
Note that $\Tr(\mb{G}^{e*})$ and
$\Tr(\mb{L}_{\mc{B}})$ have constant weighting
function $1$. Thus we can compute the extrinsic
intersection multiplicity of
$\mathrm{Trop}(\mc{L}(\bs{p}))$ and $\Tr(\Sev)$
at $\w$ as follows:
\[\begin{array}{rcl}
\xi(\w;\mathrm{Trop}(\mc{L}(\bs{p})),\Tr(\Sev))&=&\Tr(\mb{G}^{e*})\cdot\Tr(\mb{L}_{\mc{B}})\\
&=&\Tr(\mb{G}^e)\cdot\pi_*(\Tr(\mb{L}_{\mc{B}}))\\
&=&\frac {\Tr(\mb{V})\cdot\pi_*(\Tr(\mb{L}_{\mc{B}}))}{l(\mb{V})}\\
&=&\frac{\prod
2\mr{area}(\mr{Triangles})}{\tilde{\prod}length(\mr{Edges}(\D_{\w}))l(\mb{V})}
\end{array}\]
Now we prove the last statement: it follows from
the definition of $\bsm(\w)$ given in \S
\ref{sec:IT}:
\[\begin{array}{l}
\bsm(\w;\mathrm{Trop}(\mc{L}(\bs{p})),\Tr(\Sev))\\
=\bsm_{\mc{L}(\bs{p})}(\w)\cdot\bsm_{\Sev}(\w)\cdot\xi(\w;\mathrm{Trop}(\mc{L}(\bs{p})),\Tr(\Sev))
\\=1\cdot l(\mb{V})\cdot \widetilde{\prod}\mr{length(\mr{Edges})}\cdot
\frac{\prod
2\mr{area}(\mr{Triangles})}{\tilde{\prod}length(\mr{Edges}(\D_{\w}))l(\mb{V})}\\
=\prod 2\mr{area}(\mr{Triangles})
\end{array}\]
\end{proof}
Therefore from the Theorem \ref{thm:Severi
degree},  Mikhalkin's enumeration of tropical
curves is equal to the computation of the degree
$(\Tr(\mc{L}(\bs{p})\cdot\Tr(\Sev))$.

\addcontentsline{toc}{chapter}{Bibliography}
\bibliographystyle{plain}

\end{document}